%% file: arxiv.tex
\documentclass[11pt,reqno]{article}
\usepackage{fullpage,graphicx,xcolor,natbib}
\usepackage[colorlinks,citecolor=blue,urlcolor=blue,breaklinks]{hyperref}

\usepackage{amssymb}
\usepackage{amsthm} 
\usepackage{amsmath} 

\usepackage{mathtools}
\mathtoolsset{showonlyrefs}
\usepackage{xparse}

\usepackage{authblk}

\usepackage[utf8]{inputenc}
\usepackage[T1]{fontenc}
\usepackage{newtxtext}
\usepackage{booktabs}       
\usepackage{microtype}      
\usepackage{natbib}         

\usepackage{etoolbox}
\usepackage[enable]{easy-todo}

\usepackage{graphicx} 
\usepackage{float}
\usepackage{multirow}

\usepackage{subfig}

\usepackage{algorithmic,algorithm}

\let \oldsection \section
\renewcommand{\section}{\vspace{3ex plus 1ex}\oldsection}

\usepackage{empheq}

\input defs.tex

\title{Projection-Free Optimization on Uniformly Convex Sets}

\author[1,3]{Thomas Kerdreux}
\author[2,3]{Alexandre d'Aspremont}
\author[4,5]{Sebastian Pokutta}
\affil[1]{INRIA, Paris} \affil[2]{CNRS, UMR 8548}
\affil[3]{D.I., \'Ecole Normale Sup\'erieure, Paris, France}
\affil[4]{Zuse Institute, Berlin, Germany}
\affil[5]{Technische Universit{\"a}t, Berlin, Germany}

\date{\today}

\begin{document}

\maketitle
\begin{abstract}
The Frank-Wolfe method solves smooth constrained convex optimization problems at a generic sublinear rate of $\mathcal{O}(1/T)$, and it (or its variants) enjoys accelerated convergence rates for two fundamental classes of constraints: polytopes and strongly-convex sets. Uniformly convex sets non-trivially subsume strongly convex sets and form a large variety of \textit{curved} convex sets commonly encountered in machine learning and signal processing. For instance, the $\ell_p$-balls are uniformly convex for all $p > 1$, but strongly convex for $p\in]1,2]$ only. We show that these sets systematically induce accelerated convergence rates for the original Frank-Wolfe algorithm, which continuously interpolate between known rates. 
Our accelerated convergence rates emphasize that it is the curvature of the constraint sets -- not just their strong convexity -- that leads to accelerated convergence rates. 
These results also importantly highlight that the Frank-Wolfe algorithm is adaptive to much more generic constraint set structures, thus explaining faster empirical convergence.
Finally, we also show accelerated convergence rates when the set is only locally uniformly convex and provide similar results in online linear optimization.
\end{abstract}

\section{Introduction}
The Frank-Wolfe method \citep{frank1956algorithm} (Algorithm \ref{algo:FW_general}) is a projection-free algorithm designed to solve
\BEQ\label{eq:opt_problem}\tag{OPT}
\underset{x\in\mathcal{C}}{\text{argmin }} f(x),
\EEQ
where $\mathcal{C}$ is a compact convex set and $f$ a smooth convex function.
Many recent algorithmic developments in this family of methods are motivated by appealing properties already contained in the original Frank-Wolfe algorithm.
Each iteration requires to solve a Linear Minimization Oracle (see line \ref{line:FW_LMO} in Algorithm \ref{algo:FW_general}), instead of a projection or proximal operation that is not computationally competitive in various settings.
Also, the Frank-Wolfe iterates are convex combinations of extreme points of $\mathcal{C}$, the solutions of the Linear Minimization Oracle. 
Hence, depending on the extremal structure of $\mathcal{C}$, early iterates may have a specific structure, being, \textit{e.g.} , sparse or low rank for instance, that could be traded-off with the iterate approximation quality of problem \eqref{eq:opt_problem}.
These fundamental properties are among the main features that contribute to the recent revival and extensions of the Frank-Wolfe algorithm \citep{Clar10,Jagg11} used for instance in large-scale structured prediction \citep{bojanowski2014weakly,bojanowski2015weakly,alayrac2016unsupervised,seguin2016instance,miech2017learning,peyre2017weakly,miech2018learning}, quadrature rules in RKHS \citep{bach2012equivalence,lacoste2015sequential,futami2019bayesian}, optimal transport \citep{courty2016optimal,vayer2018optimal,paty2019subspace,luise2019sinkhorn}, and many others.

\begin{algorithm}[h!]
  \caption{Frank-Wolfe Algorithm}\label{algo:FW_general}
  \begin{algorithmic}[1]
    \REQUIRE $x_0 \in \mathcal{C}$, $L$ upper bound on the Lipschitz constant.
    \FOR{$t=0, 1, \ldots, T $}
    	\STATE $ v_t \in \underset{v\in\mathcal{C}}{\text{argmax }} \langle -\nabla f(x_t); v - x_t\rangle$ \hfill $\triangleright$ \text{Linear minimization oracle} \label{line:FW_LMO}
    	
        \STATE $\gamma_t = \underset{\gamma\in[0,1]}{\text{argmin }} \gamma \langle v_t - x_t; \nabla f(x_t)\rangle + \frac{\gamma^2}{2} L ||v_t - x_t||^2$ \hfill $\triangleright$ \text{Short step} \label{line:step_size}
        
        \STATE $x_{t+1} = (1 - \gamma_t)x_{t} + \gamma_t v_{t}$ \hfill $\triangleright$ \text{Convex update} \label{line:FW_update}
    \ENDFOR
  \end{algorithmic}
\end{algorithm}

\paragraph{Uniform Convexity.}
Uniform convexity is a global quantification of the curvature of a convex set~$\mathcal{C}$. There exists several definitions, see for instance, \citep[Theorem 2.1.]{goncharov2017strong} and \citep{abernethy2018faster,Molinaro2020} for the strongly convex case. Here, we focus on the generalization of a classic definition of the strong convexity of a set \citep{garber2015faster}.

\begin{definition}[$\gamma$ uniform convexity of $\mathcal{C}$]\label{def:set_uniform_convexity}
A closed set $\mathcal{C}\subset\mathbb{R}^d$ is $\gamma_{\mathcal{C}}$-uniformly convex with respect to a norm  $||\cdot||$, if for any $x,y\in\mathcal{C}$, any $\eta\in[0,1]$ and any $z\in\mathbb{R}^d$ with $||z|| = 1$, we have
\[
\eta x + (1-\eta)y + \eta (1 - \eta ) \gamma_{\mathcal{C}}(||x-y||) z \in \mathcal{C},
\]
where $\gamma_{\mathcal{C}}(\cdot)\geq 0$ is a non-decreasing function. In particular when there exists $\alpha >0$ and $q >0$ such that $\gamma_{\mathcal{C}}(r)\geq \alpha r^q$, we say that $\mathcal{C}$ is $(\alpha, q)$-uniformly convex or $q$-uniformly convex.
\end{definition}

The uniform convexity assumption strengthens the convexity property of $\mathcal{C}$ that any line segment between two points is included in $\mathcal{C}$. It requires a scaled unit ball to fit in $\mathcal{C}$ and results in curved sets.
Strongly convex sets are uniformly convex sets for which $\gamma_{\mathcal{C}}(r)\geq\alpha r^2$, \textit{i.e.} $(\alpha, 2)$-uniformly convex sets.
Two common families of uniformly convex sets are the $\ell_p$-balls and $p$-Schatten balls which are uniformly convex for any $p>1$ but strongly convex for $p\in]1,2]$ only, \textit{i.e.} $2$-uniformly convex sets for $p\in]1,2]$.

\paragraph{Convergence Rates for Frank-Wolfe.}
The Frank-Wolfe algorithm admits a tight \citep{canon1968tight,jaggi2013revisiting,lan2013complexity} general sublinear convergence rate of $\mathcal{O}(1/T)$ when $\mathcal{C}$ is a compact convex set and $f$ is a convex $L$-smooth function.
However, when the constraint set $\mathcal{C}$ is strongly-convex and $\text{inf}_{x\in\mathcal{C}}||\nabla f(x)|| > 0$, Algorithm~\ref{algo:FW_general} enjoys a linear convergence rate \citep{levitin1966constrained,demyanov1970}. Later on, the work of \citep{dunn1979rates} showed that linear rates are maintained when the constraint set satisfies a condition subsuming local strong-convexity.
Interestingly, this linear convergence regime \emph{does not} require the strong-convexity of $f$, \textit{i.e.} the lower quadratic additional structure comes from the constraint set rather than from the function. When $x^*$ is in the interior of $\mathcal{C}$ and $f$ is strongly convex, Algorithm \ref{algo:FW_general} also enjoys a linear convergence rate \citep{guelat1986some}.

These two linear convergence regimes can both become arbitrarily bad as $x^*$ gets close to the border of $\mathcal{C}$, and do not apply in the limit case where the unconstrained optimum of $f$ lies at the boundary of $\mathcal{C}$.
In this scenario, when the constraint set is strongly convex, \cite{garber2015faster} prove a general sublinear rate of $\mathcal{O}(1/T^2)$ when $f$ is $L$-smooth and $\mu$-strongly convex.
In early iterations, these convergence rates can beat badly-conditioned linear rates.

Other structural assumptions are known to lead to accelerated convergence rates. However, these require elaborate algorithmic enhancements of the original Frank-Wolfe algorithm. Polytopes received much attention in particular, with \textit{corrective} or \textit{away} algorithmic mechanisms \citep{guelat1986some,hearn1987restricted} that lead to linear convergence rates under appropriate structures of the objective function \citep{garber2013linearly,lacoste2013affine,lacoste2015global,beck2017linearly,gutman2018condition,pena2018polytope}. 
Accelerated versions of Frank-Wolfe, when the constraint set is a trace-norm ball (a.k.a. nuclear balls) -- which are neither polyhedral nor strongly convex \citep{so1990facial} -- have also received a lot of attention \citep{freund2017extended,allen2017linear,garber2018fast} and are especially useful in matrix completion \citep{jaggi2010simple,shalev2011large,harchaoui2012large,dudik2012lifted}.

\paragraph{Contributions.} 
We show accelerated sublinear convergence rates for the Frank-Wolfe algorithm, with appropriate line-search, for smooth constrained optimization problems when the constraint set is globally or locally uniformly convex. 
These bounds generalize the rates of \citep{polyak1966existence,demyanov1970}, \citep{dunn1979rates}, and \citep{garber2015faster} in their respective settings and fill the gap between all known convergence rates, \textit{i.e.} between $\mathcal{O}(1/T)$ and the linear rate of \citep{levitin1966constrained,demyanov1970,dunn1979rates}, and between $\mathcal{O}(1/T)$ and the $\mathcal{O}(1/T^2)$ rate of \citep{garber2015faster} (see \textit{e.g.} concluding remarks of \citep{garber2015faster}). 
We also provide similar arguments that interpolate between known regret bounds in an example of projection-free online learning. 
Overall, we illustrate another key aspect of the Frank-Wolfe algorithms: they are adaptive to many generic structural assumptions.

\paragraph{Outline.}
In Section~\ref{sec:analysis}, we analyze the complexity of the Frank-Wolfe algorithm when the constraint set is uniformly convex, under various assumptions on $f$.
In Section \ref{ssec:local_convergence_rates}, we also establish accelerated convergence rate under weaker assumptions than global or local uniform convexity of the constraint set.
In Section~\ref{sec:online_optimization}, we focus on the online optimization setting and provide analogous results to the previous section in term of regret bounds. 
In Section~\ref{sec:example_uniform_convexity}, we give some examples of uniformly convex sets and relate the uniform convexity notion for sets with that of spaces and functions. 

\paragraph{Notation.}
We use $d$ for the \emph{ambient dimension} of the compact convex sets $\mathcal{C}$. 
We denote the \emph{boundary} of $\mathcal{C}$ by $\partial\mathcal{C}$ and let 
$N_{\mathcal{C}}(x)\triangleq \{d~|~\langle d; y - x\rangle \leq 0,~\forall y\in\mathcal{C}\}$ denote the \emph{normal cone} at $x$ with respect to $\mathcal{C}$. In the following, $x^*$ is an (optimal) solution to \eqref{eq:opt_problem} and $(\alpha, q)$ denotes the uniform convexity parameters of a set. $p$ stands for the parameters for the various norm balls and might differ from $q$. We sometimes assume strict convexity of $f$ for the sake of exposition (only).
Given a norm $||\cdot||$ we denote by $||d||_* \triangleq \text{max}_{||x||\leq 1} \langle x; d \rangle$ its dual norm and we let $h_t \triangleq f(x_t) - f(x^*)$ denote the primal gap.

\section{Frank-Wolfe Convergence Analysis with Uniformly Convex Constraints}\label{sec:analysis}
In Theorem \ref{th:uniformly_cvx_set_rates_general}, we show accelerated convergence rate of the Frank-Wolfe algorithm when the constraint set $\mathcal{C}$ is $(\alpha, q)$-uniformly convex (with $q \geq 2$) and the smooth convex function satisfies $\text{inf}_{x\in\mathcal{C}}||\nabla f(x)||> 0$; this is the interesting case.
In Section~\ref{ssec:local_convergence_rates}, we then explore \emph{localized} uniform convexity on the set $\mathcal{C}$ and provide convergence rates in Theorem~\ref{th:uniformly_cvx_set_rates_general_local}. In Theorem \ref{th:rates_uniform_convexity} we show that $(\alpha, q)$-uniform convexity ensures convergence rates of the Frank-Wolfe algorithms in between the $\mathcal{O}(1/T)$ and $\mathcal{O}(1/T^2)$ \citep{garber2015faster} when the function is strongly convex (and $L$-smooth), or satisfies a quadratic error bound at $x^*$. We also provide generalized convergence rates assuming H\"olderian Error Bounds on $f$. 
In all of these scenarios, when the set is uniformly convex, the Frank-Wolfe algorithm (with short step) enjoys accelerated convergence rates with respect to $\mathcal{O}(1/T)$.

\paragraph{Proof Sketch.}
We now provide an informal discussion as to why the uniform convexity of $\mathcal{C}$ leads to accelerated convergence rates under the classical assumptions that $\text{inf}_{x\in\mathcal{C}} ||\nabla f(x)|| > 0$ and hence $x^*\in\partial\mathcal{C}$. Formal arguments are developed in the proof of Theorem~\ref{th:uniformly_cvx_set_rates_general}.
The key point is that if $\mathcal{C}$ is curved around $x^*$ and $f$ is $L$-smooth, when $||x_t - x^*||$ converges to zero, the quantity $||x_t - v_t||$ also converges to zero, which is generally not the case, for instance when the constraint set is a polytope.

In Figure \ref{fig:5_uniform_convexity_geometric_explanation} we show various such behaviors. Applying the $L$-smoothness of $f$ to the Frank-Wolfe iterates, the classical iteration inequality is of the form (with $\gamma\in[0,1]$)
\BEQ\label{eq:L_smoothness_trade_off}
f(x_{t+1}) - f(x^*) \leq f(x_t) - f(x^*) - \gamma \langle - \nabla f(x_t); v_t - x_t\rangle + \frac{\gamma^2}{2} L ||x_t - v_t||^2.
\EEQ
The non-negative quantity $\langle - \nabla f(x_t); v_t - x_t\rangle$ participates in guaranteeing the function decrease, counter-balanced with $||x_t - v_t||^2$.
The convergence rate then depends on specific relative quantification of these various terms, that we call scaling inequalities in Lemma~\ref{lemma:global_scaling_UC} and \ref{lem:simple_local_UC_imply_local_scaling}.

\begin{figure}[h!]
    \centering
    \includegraphics[width=0.80\linewidth]{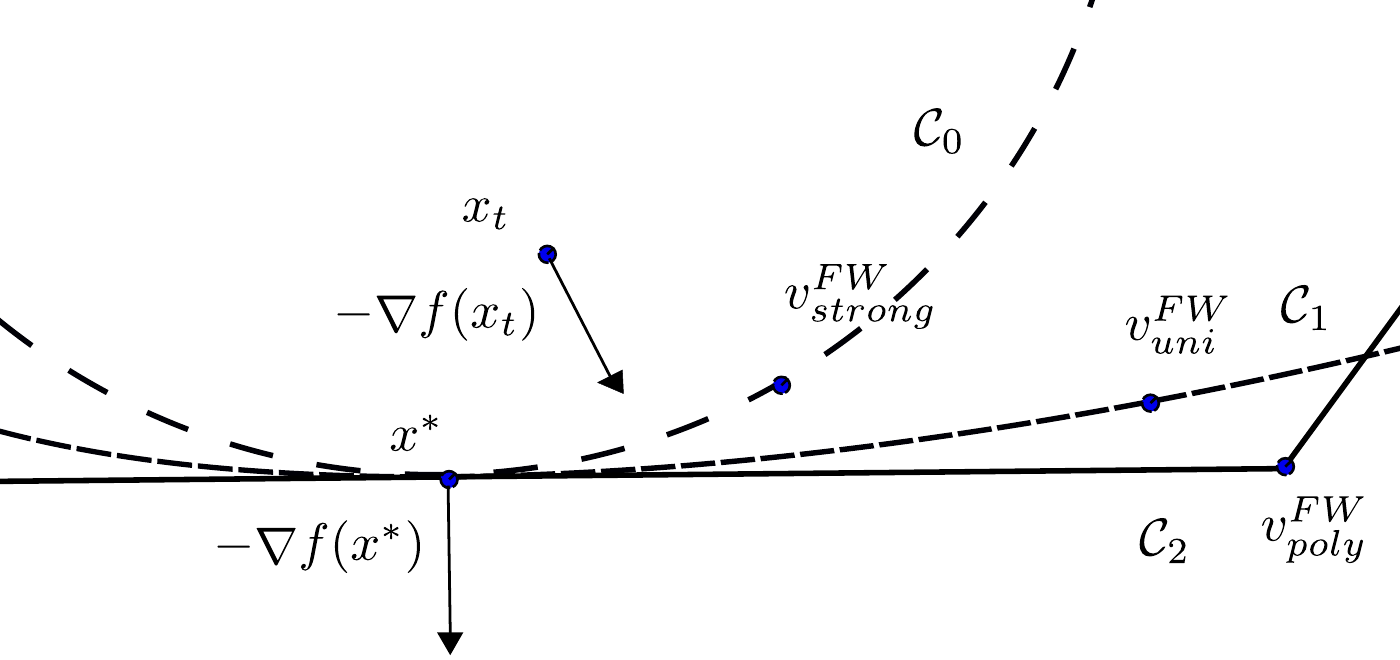}
    \caption{$v^{FW}_{strong}$, $v^{FW}_{uni}$, $v^{FW}_{poly}$ represent the various FW vertices from the strongly convex set $\mathcal{C}_0$, the uniformly convex set $\mathcal{C}_1$ and the polytope $\mathcal{C}_2$.}\label{fig:5_uniform_convexity_geometric_explanation}
\end{figure}

\subsection{Scaling Inequality on Uniformly Convex Sets}\label{ssec:scaling_inequalities}
The following lemma outlines that the uniform convexity of $\mathcal{C}$ implies an upper bound on the distance between the current iterate and the Frank-Wolfe vertex as a power of the Frank-Wolfe gap.
Note that the uniform convexity is defined with respect to any norm, and not just in terms of an Hilbertian structure. 
To be even more generic, the uniform convexity can be defined with respect to gauge functions that are not necessarily norms, see, for instance, the strong-convexity of \citep{Molinaro2020}.

\begin{lemma}\label{lemma:global_scaling_UC}
Assume the compact $\mathcal{C}\subset\mathbb{R}^d$ is an $(\alpha, q)$-uniformly convex set with respect to a norm $||\cdot||$, with $\alpha > 0$ and $q\geq 2$. Consider $x\in\mathcal{C}$, $\phi\in\mathbb{R}^d$ and $v_{\phi}\in\argmax_{v\in\mathcal{C}}{\langle \phi; v \rangle}$. Then, we have $\langle \phi; v_{\phi} - x\rangle \geq \frac{\alpha}{2} ||v_{\phi} - x||^q||\phi||_*$.
In particular for an iterate $x_t$ and its associated Frank-Wolfe vertex $v_t$, this yields
\BEQ\label{eq:scaling_UC}\tag{Global-Scaling}
\langle -\nabla f(x_t); v_t - x_t\rangle \geq \frac{\alpha}{2} ||v_t - x_t||^q||\nabla f(x_t)||_*.
\EEQ
\end{lemma}
\begin{proof}
Because $\mathcal{C}$ is $(\alpha, q)$-uniformly convex, we have that for any $z\in\mathbb{R}^d$ of unit norm $(x + v_{\phi})/2 + \alpha/4 ||x - v_{\phi}||^q z \in\mathcal{C}$.
By optimality of $v_{\phi}$, we have
$\langle \phi; v_{\phi}\rangle  \geq  \langle \phi; (x + v_{\phi})/2\rangle + \alpha/4 ||x - v_{\phi}||^q \langle \phi, z\rangle$. 
Hence, choosing the best $z$ implies $\langle \phi; v_{\phi} - x\rangle \geq \alpha/2 ||v_{\phi} - x||^q ||\phi||_*$.
\end{proof}

In other words, when $\mathcal{C}$ is uniformly convex, \eqref{eq:scaling_UC} quantifies the trade-off between the Frank-Wolfe gap $g(x_t) \triangleq \langle \nabla f(x_t); x_t - v_t\rangle$ and the value of $||x_t - v_t||$ under consideration in \eqref{eq:L_smoothness_trade_off}.

\subsection{Interpolating linear and sublinear rates}\label{ssec:sublinear_linear}

To our knowledge, no accelerated convergence rate of the Frank-Wolfe algorithm is known when the constraint set is uniformly convex but not strongly convex. We fill this gap in Theorem~\ref{th:uniformly_cvx_set_rates_general} below. When $q$ goes to $+\infty$, we recover the classic sublinear convergence rate of $\mathcal{O}(1/T)$.

\begin{theorem}\label{th:uniformly_cvx_set_rates_general}
Consider a convex $L$-smooth function $f$ and a compact convex set $\mathcal{C}$. Assume that $\mathcal{C}$ is $(\alpha, q)$-uniformly convex set with respect to a norm $||\cdot||$, with $q \geq 2$. 
Assume $||\nabla f(x)||_* \geq c > 0 $ for all $x\in\mathcal{C}$.
Then the iterates of the Frank-Wolfe algorithm, with short step as in Line~\ref{line:step_size} of Algorithm \ref{algo:FW_general} or exact line search, satisfy
\begin{equation}\label{eq:global_rate_demyanov}
  \left\{
    \begin{split}
    f(x_T) - f(x^*) &\leq M/(T + k)^{1/(1-2/q)}~~\text{ when }  q > 2 \\
    f(x_T) - f(x^*) &\leq \big( 1 - \rho \big)^T h_0 ~~\text{ when } q=2,
    \end{split}
  \right.
\end{equation}
with $\rho = \text{max}\big\{ \frac{1}{2}, 1 - c\alpha/L \big\}$, $k \triangleq (2 - 2^{\eta})/({2^\eta -1})$ and $M \triangleq \text{max}\{h_0 k^{1/\eta}, 2/((\eta-(1-\eta)(2^\eta -1))C)^{1/\eta}\}$, where $\eta \triangleq 1 - 2/q$ and $C \triangleq (c\alpha/2)^{2/q}/(2L)$.
\end{theorem}

\begin{proof}
By $L$-smoothness of $f$ and because of the short step, we have for $\gamma\in[0,1]$
\[
f(x_{t+1}) \leq f(x_t) - \gamma g(x_t) + \frac{\gamma^2}{2} L ||x_t - v_t||^2,
\]
where $g(x_t)$ is the Frank-Wolfe gap. With $\gamma = \text{min}\big\{ 1, g(x_t)/(L ||x_t - v_t||^2)\big\}$ we have
\[
f(x_{t+1}) \leq f(x_t) - \frac{g(x_t)}{2} \cdot \text{min}\Big\{ 1; \frac{g(x_t)}{L ||x_t - v_t||^2}\Big\}.
\]
Applying Lemma~\ref{lemma:global_scaling_UC} with $\phi=-\nabla f(x_t)$ gives $g(x_t) \geq \alpha/2 ||x_t-v_t||^q ||\nabla f(x_t)||_*$. Then
\BEQ\label{eq:central_equation}
\frac{g(x_t)}{||x_t - v_t||^2} = \Big( \frac{g(x_t)^{q/2-1} g(x_t)}{||x_t - v_t||^q}\Big)^{2/q} \geq \Big(\alpha/2 ||\nabla f(x_t)||_*\Big)^{2/q} g(x_t)^{1 - 2/q}.
\EEQ
Finally, because $g(x_t)\geq f(x_t) - f(x^*) = h(x_t)$, we have
\[
h(x_{t+1}) \leq h(x_t) - \frac{h(x_t)}{2} \text{min}\Big\{ 1; \Big(\alpha/2 ||\nabla f(x_t)||_*\Big)^{2/q} h(x_t)^{1-2/q} /L \Big\},
\]
and hence
\BEQ\label{eq:principal_result}
h(x_{t+1}) \leq h(x_t) \cdot \text{max}\Big\{ \frac{1}{2}; 1 - \Big(\alpha/2 ||\nabla f(x_t)||_*\Big)^{2/q} h(x_t)^{1 - 2/q} /(2 L) \Big\}.
\EEQ
Then, by assumption, for all $x\in\mathcal{C}$, we have $||\nabla f(x)||_* > c > 0$ and hence \eqref{eq:principal_result} becomes
\[
h(x_{t+1}) \leq h(x_t) \cdot \text{max}\Big\{ \frac{1}{2}; 1 - \big(c\alpha/2\big)^{2/q} h(x_t)^{1 - 2/q} /(2 L) \Big\}.
\]
We solve the recursion with Lemma~\ref{lem:general_recursive_equation}; when $q=2$ we recover the linear convergence rate.
\end{proof}

\begin{remark}
The convergence rates in Theorem~\ref{th:uniformly_cvx_set_rates_general} imply convergence rates in term of distance to optimum by applying Lemma~\ref{lemma:global_scaling_UC} with $\phi=-\nabla f(x^*)$ and convexity of $f$. Indeed, this yields
\[
||x_t - x^*||^q \leq \frac{2}{c\alpha} \langle -\nabla f(x^*) ; x^* - x_t \rangle \leq \frac{2}{c\alpha} \big( f(x_t) - f(x^*) \big).
\]
Hence, to obtain convergence rates in terms of the distance of the iterates to the optimum, the uniform convexity of the set supersedes that of the function, which is not needed here. 
\end{remark}

\subsection{Convergence Rates with Local Uniform Convexity}\label{ssec:local_convergence_rates}
Theorem~\ref{th:uniformly_cvx_set_rates_general} relies on the global uniform convexity of the set. Actually, for the strongly convex case, it is equivalent to the global scaling inequality~\eqref{eq:scaling_UC}, see \citep[Theorem~2.1 (g)]{goncharov2017strong}.
However, weaker assumptions also lead to accelerated convergence rates of the Frank-Wolfe algorithm. In Theorem~\ref{th:uniformly_cvx_set_rates_general_local}, we show accelerated convergence rates assuming a \emph{local} scaling inequality at $x^*$.
We then study the sets for which such an inequality holds. 
We say that a local scaling inequality holds at $x^*\in\mathcal{C}$, when there exists an $\alpha>0$ and $q\geq 2$ such that for all $x\in\mathcal{C}$
\BEQ\label{eq:local_scaling_inequality}\tag{Local-Scaling}
\langle -\nabla f(x^*); x^* - x\rangle \geq \alpha/2 ||\nabla f(x^*)||_* \cdot ||x^* - x||^q.
\EEQ
This combines the position of $-\nabla f(x^*)$ with respect to the normal cone of $\mathcal{C}$ at $x^*$ and the local geometry of $\mathcal{C}$ at $x^*$, see Remark~\ref{cor:more_than_UC}.
When the set $\mathcal{C}$ is globally $(\alpha, q)$-uniformly convex, this is a direct consequence of Lemma~\ref{lemma:global_scaling_UC} because $-\nabla f(x^*)\in N_{\mathcal{C}}(x^*)$.
In the following lemma, we prove that it is also a consequence of a natural definition of local uniform convexity of $\mathcal{C}$ at $x^*$.
A proof is given in Appendix~\ref{ssec:proof_scaling_inequalities}.

\begin{lemma}\label{lem:simple_local_UC_imply_local_scaling}
Consider a compact convex set $\mathcal{C}$ and $x^*$ a solution to \eqref{eq:opt_problem}.
Assume that $\mathcal{C}$ is locally $(\alpha, q)$-uniformly convex at $x^*$ with respect to $||\cdot||$ in the sense that, for all $x\in\mathcal{C}$, $\eta\in[0,1]$ and unit norm $z\in\mathbb{R}^d$, we have $\eta x^* + (1-\eta)x + \eta (1 - \eta ) \alpha ||x^* - x||^q z \in \mathcal{C}$. 
Then \eqref{eq:local_scaling_inequality} holds at $x^*$ with parameters $(\alpha, q)$.
\end{lemma}
\begin{proof}
By definition of local uniform convexity between $x^*$ and $x$, we have that for any $z\in\mathbb{R}^d$ of unit norm $(x^* + x)/2 + \alpha/4 ||x^* - x||^q z \in\mathcal{C}$. Then, by optimality of $x^*$, \textit{i.e.} $x^*\in\text{argmax}_{v\in\mathcal{C}} \langle -\nabla f(x^*); v\rangle$, we have
$\langle -\nabla f(x^*); x^*\rangle  \geq  \langle -\nabla f(x^*); (x^* + x)/2\rangle + \alpha/4 ||x^* - x||^q \langle -\nabla f(x^*), z\rangle$. Choosing the best $z$ and subtracting both sides by $\langle -\nabla f(x^*); x\rangle$, implies
\[
\langle -\nabla f(x^*); x^* - x\rangle \geq \alpha/2 ||x^* - x||^q ||\nabla f(x^*)||_*.
\]
\end{proof}

We obtain sublinear convergence rates that are systematically better than the $\mathcal{O}(1/T)$ baseline for any $q\geq 2$.

\begin{theorem}\label{th:uniformly_cvx_set_rates_general_local}
Consider $f$ an $L$-smooth convex function and a compact convex set $\mathcal{C}$. Assume $||\nabla f(x)||_* > c > 0$ for all $x\in\mathcal{C}$ and write $x^*\in\partial\mathcal{C}$ a solution of \eqref{eq:opt_problem}. Further, assume that the convex set $\mathcal{C}$ satisfies a local scaling inequality at $x^*$ with parameters $(\alpha, q)$.
Then the iterates of the Frank-Wolfe algorithm, with short step satisfy
\begin{equation}
  \left\{
    \begin{split}
    f(x_T) - f(x^*) &\leq M/(T + k)^{\frac{1}{1-2/(q(q-1))}}~~\text{ when }  q > 2 \\
    f(x_T) - f(x^*) &\leq \big( 1 - \rho \big)^T h_0 ~~\text{ when } q=2,
    \end{split}
  \right.
\end{equation}
with $\rho = \text{max}\big\{ \frac{1}{2}, 1 - c\alpha/L \big\}$, $k \triangleq (2 - 2^{\eta})/({2^\eta -1})$ and $M \triangleq \text{max}\{h_0 k^{1/\eta}, 2/((\eta-(1-\eta)(2^\eta -1))C)^{1/\eta}\}$, where $\eta \triangleq 1 - 2/(q(q-1))$ and $C \triangleq 1/(2LH^2)$. Note that $H$ depends only on $C, \alpha,  L$ and $q$ (see Lemma \ref{lemma:local_scaling_inequality_consequence}).
\end{theorem}

\begin{remark}
When the local scaling inequality \eqref{eq:local_scaling_inequality} holds with $q=2$, we obtain the same linear convergence regime as in \eqref{eq:global_rate_demyanov}. 
With $q>2$, the sublinear convergence rates are of order $\mathcal{O}(1/T^{1/(1-2/(q(q-1)))})$ instead of $\mathcal{O}(1/T^{1/(1-2/q)})$ when the set is $(\alpha, q)$-uniformly convex and the global scaling inequality~\eqref{eq:scaling_UC} holds. It is an open question to close this gap in the convergence regime with the local scaling inequality only.
\end{remark}

The local scaling inequality expresses a property between $x^*$ and any $x\in\mathcal{C}$.
In the following lemma, we show that albeit we only have access to a local scaling inequality, it is still possible to control the variation of the distance of the iterate to its Frank-Wolfe vertex $||x_t - v_t||$ in terms of a power of the primal gap, see beginning of Section \ref{sec:analysis} for a qualitative explanation. This is key for the proof of Theorem \ref{th:uniformly_cvx_set_rates_general_local}.

\begin{lemma}\label{lemma:local_scaling_inequality_consequence}
Consider $f$ a $L$-smooth convex function and a compact convex set $\mathcal{C}$. Assume $\text{inf}_{x\in\mathcal{C}}||\nabla f(x)||_* > c > 0$ and write $x^*\in\partial\mathcal{C}$ the solution of \eqref{eq:opt_problem}. Assume that $\mathcal{C}$ satisfies a local scaling inequality at $x^*$ for problem~\eqref{eq:opt_problem} with $\alpha > 0$ and $q\geq 2$, \textit{i.e.} for all $x\in\mathcal{C}$
\BEQ\label{eq:local_scaling_lemma}
\langle -\nabla f(x^*); x^* - x\rangle \geq \alpha/2 ||\nabla f(x^*)||_* \cdot ||x^* - x||^q 
\EEQ
Write $v_t \triangleq \text{argmax}_{v\in\mathcal{C}} \langle -\nabla f(x_t); v\rangle$ the Frank-Wolfe vertex. Assume that $h_t=f(x_t) - f(x^*)\leq 1$ (a simple burn-in phase). Then, we have
\BEQ
||x_t - v_t|| \leq H h_t^{1/(q(q-1))},
\EEQ
with $H \triangleq 2 \cdot \text{max}\Big\{\Big(\frac{2L}{c \alpha}\Big)^{1/(q-1)} \Big(\frac{2}{c\alpha}\Big)^{1/(q(q-1))}, \Big(\frac{2}{c\alpha}\Big)^{1/q}\Big\}$.
\end{lemma}
\begin{proof}
We apply the local scaling inequality \eqref{eq:local_scaling_lemma} with $x=v_t$ and $x=x_t$.
We obtain two important inequalities: one that upper bounds $||x - x^*||$ in terms of $f(x)-f(x^*)$ and another that upper bounds $||v_t - x^*||$ in terms of $||x^* - x_t||$, where $v_t$ is the Frank-Wolfe vertex related to iterate $x_t$. These two inequalities rely of convexity, $L$-smoothness and \eqref{eq:local_scaling_lemma}, but do not rely on strong convexity of the function $f$.

By optimality of the Frank-Wolfe vertex $v_t$, we have $\nabla f(x_t)^T v_t \leq \nabla f(x_t)^T x^*$. Hence, combining that with Cauchy-Schwartz, we get
\begin{eqnarray*}
||\nabla f(x^*) - \nabla f(x_t)|| ~||v_t - x^*|| &\geq& \langle \nabla f(x^*) - \nabla f(x_t); v_t - x^*\rangle + \underbrace{\langle \nabla f(x_t); v_t - x^*\rangle}_{\leq 0}\\
&\geq& \langle \nabla f(x^*) ; v_t - x^*\rangle \geq c \alpha/2 ||v_t - x^*||^q.
\end{eqnarray*}
Then, $L$-smoothness applied to the left hand side leaves us with
\BEQ\label{eq:FW_dir_bound}
||x_t - x^*|| \geq \frac{c \alpha}{2L} ||v_t - x^*||^{q-1},
\EEQ
and a triangular inequality gives
\begin{eqnarray*}
||x_t - v_t|| & \leq & ||v_t - x^*|| + ||x^* - x_t||\\
||x_t - v_t|| & \leq & \Big(\frac{2L}{c \alpha}\Big)^{1/(q-1)} ||x_t - x^*||^{1/(q-1)} + ||x^* - x_t||.
\end{eqnarray*}
Finally applying \eqref{eq:local_scaling_lemma} with $x=x_t$ and using that $\text{inf}_{x\in\mathcal{C}} ||\nabla f(x)||_* > c > 0$, we have $||x_t - x^*|| \leq \Big(\frac{2}{c\alpha}\Big)^{1/q} h_t^{1/q}$ which leads to
\[
||x_t - v_t|| \leq \Big(\frac{2L}{c \alpha}\Big)^{1/(q-1)} \Big(\frac{2}{c\alpha}\Big)^{1/(q(q-1))} h_t^{1/(q(q-1))} + \Big(\frac{2}{c\alpha}\Big)^{1/q} h_t^{1/q} ~.
\]
We can simplify this previous expression, and we assumed without loss of generality (\textit{i.e.} up to a burning-phase) that $h_t\leq 1$, which implies for $q\geq 2$ that 
$h_t^{1/(q(q-1))}\geq h_t^{1/q}$.
With $H \triangleq 2 \cdot \text{max}\Big\{\Big(\frac{2 L}{c \alpha}\Big)^{1/(q-1)} \Big(\frac{2}{c\alpha}\Big)^{1/(q(q-1))}, \Big(\frac{2}{c\alpha}\Big)^{1/q}\Big\}$, we then have
\[
||x_t - v_t|| \leq H h_t^{1/(q(q-1))}~.
\]
\end{proof}

We now proceed with the proof of Theorem~\ref{th:uniformly_cvx_set_rates_general_local}.

\begin{proof}[Proof of Theorem~\ref{th:uniformly_cvx_set_rates_general_local}]
With Lemma~\ref{lemma:local_scaling_inequality_consequence}, which satisfies the assumption of Theorem~\ref{th:uniformly_cvx_set_rates_general_local}, we have
\[
||x_t - v_t|| \leq H h_t^{1/(q(q-1))},
\]
with $H \triangleq 2 \cdot \text{max}\Big\{\Big(\frac{2L}{c \alpha}\Big)^{1/(q-1)} \Big(\frac{2}{c\alpha}\Big)^{1/(q(q-1))}, \Big(\frac{2}{c\alpha}\Big)^{1/q}\Big\}$. We plug this last expression in the classical descent guarantee given by $L$-smoothness
\begin{eqnarray*}
h_{t+1} & \leq & (1 - \gamma) h_t + \gamma^2\frac{L}{2} ||v_t - x_t||^2\\
h_{t+1} &\leq & (1 - \gamma) h_t + \gamma^2\frac{L}{2} H^2 h_t^{2/(q(q-1))}~.
\end{eqnarray*}
The optimal decrease $\gamma\in[0,1]$ is $\gamma^* = \text{min}\Big\{ \frac{h_t^{1- 2/(q(q-1))}}{L H^2}, 1\Big\}$. When $\gamma^*=1$, or equivalently $h_t\geq \big(LH^2\big)^{1 - 2/(q(q-1))}$, we have $h_{t+1} \leq h_t /2$. In other words, for the very first iterations, there is a brief linear convergence regime. Otherwise, when $\gamma^*\leq 1$, we have
\BEQ
h_{t+1} \leq h_t \Big( 1 - \frac{1}{2LH^2} h_t^{1 - 2/(q(q-1))} \Big)~.
\EEQ
When $q=2$, this corresponds to the strongly convex case and we recover the classical linear-convergence regime.
We conclude using Lemma~\ref{lem:general_recursive_equation} that the rate is $\mathcal{O}\Big(1/T^{1/(1-2/(q(q-1)))}\Big)$.
\end{proof}


A similar approach appears in \citep{dunn1979rates} which introduces the following functional
\[
a_{x^*}(\sigma) \triangleq \underset{\substack{x \in\mathcal{C} \\ ||x - x^*||\geq \sigma}}{\text{inf }} \langle \nabla f(x^*); x - x^*\rangle,
\]
and shows than when there exists $A>0$ such that $a_{x^*}(\sigma) \geq A ||x - x^*||^2$, then the Frank-Wolfe algorithm converges linearly, under appropriate line-search rules. 
This result of \citep{dunn1979rates} thus subsumes that of \citep{levitin1966constrained,demyanov1970}. 
However, no analysis was conducted for uniformly (but not strongly) convex set.

In Lemma~\ref{lem:simple_local_UC_imply_local_scaling} we showed that a given quantification of local uniform convexity implies the local scaling inequality and hence accelerated convergence rates. However, there are many situations where such a local notion of uniform convexity does not hold but \eqref{eq:local_scaling_inequality} does. This was the essence of \citep[Remark~3.5.]{dunn1979rates} that we state here.

\begin{corollary}\label{cor:more_than_UC}
Assume there exists a compact and $(\alpha, q)$-uniformly convex set $\Gamma$ such that $\mathcal{C}\subset\Gamma$ and $N_{\Gamma}(x^*)\subset N_{\mathcal{C}}(x^*)$, where $x^*$ is the solution of \eqref{eq:opt_problem}. If $-\nabla f(x^*)\in N_{\Gamma}(x^*)$, then
\eqref{eq:local_scaling_inequality} holds at $x^*$ with the $(\alpha, q)$ parameters.
\end{corollary}
\begin{proof}
Here, because $N_{\Gamma}(x^*)\subset N_{\mathcal{C}}(x^*)$, we have that $x^*\in\text{argmax}_{v\in\Gamma} \langle -\nabla f(x^*); v\rangle$. Also, for $x\in\mathcal{C}\subset\Gamma$, by $(\alpha, q)$-uniform convexity of $\Gamma$, we also have that for any $z\in\mathbb{R}^d$ of unit norm that $(x^* + x)/2 + \alpha/4 ||x^* - x||^q z \in\Gamma$. Then, by optimality of $x^*$, we have
$\langle -\nabla f(x^*); x^*\rangle  \geq  \langle -\nabla f(x^*); (x^* + x)/2\rangle + \alpha/4 ||x^* - x||^q \langle -\nabla f(x^*), z\rangle$. Choosing the best $z$ and subtracting both sides by $\langle -\nabla f(x^*); x\rangle$, implies (for any $x\in\mathcal{C}$)
$\langle -\nabla f(x^*); x^* - x\rangle \geq \alpha/2 ||x^* - x||^q ||\nabla f(x^*)||_*$.
\end{proof}

There exist numerous notions of local uniform convexity of a set that may imply local scaling inequalities. See for instance, the local directional strong convexity in \citep[\S Local Strong Convexity]{goncharov2017strong}.
Alternatively, in the context of functions, H\"olderian Errors Bounds (HEB) offer a weaker description of localized uniform convexity assumptions while retaining the same convergence rates \citep{kerdreux2019restarting}. 
And these are known to hold generically for various classes of function \citep{law1965ensembles,kurdyka1998gradients,bolte2007lojasiewicz}. 
Obtaining a similar characterization for set is of interest.
In particular, it is natural to relate enhanced convexity properties of the set gauge function $||\cdot||_{\mathcal{C}}$ \citep[\S 15]{Rock70} to convexity properties of the set or directly to local scaling inequalities. 
For instance, local uniform convexity of the gauge $||\cdot||_{\mathcal{C}}$ implies a local scaling inequality for $\mathcal{C}$ (see Lemma \ref{lem:UC_gauge_to_scaling}). 
This suggests that error bounds as guaranteed with {\L}ojasiewicz-type arguments on the gauge function should imply local scaling inequalities, showing that theses inequalities hold somewhat generically.

\subsection{Interpolating Sublinear Rates for Arbitrary $x^*$}\label{ssec:sublinear_sublinear}
When the function is $\mu$-strongly convex and the set $\mathcal{C}$ is $\alpha$-strongly convex, \cite{garber2015faster} show that the Frank-Wolfe algorithm (with short step) enjoys a general $\mathcal{O}(1/T^2)$ convergence rate. 
In particular, this result does not depend on the location of $x^*$ with respect to $\mathcal{C}$. We now generalize this result by relaxing the strong convexity of the constraint set $\mathcal{C}$ and the quadratic error bound on $f$ \citep[(1)]{garber2015faster}.

\paragraph{H\"olderian Error Bounds.}
Let $f$ be a strictly convex $L$-smooth function and $x^*=\text{argmin}_{x\in\mathcal{C}}f(x)$ where $\mathcal{C}$ is a compact convex set; the strict convexity assumption is only required to simplify exposition and the results hold more generally with the usual generalizations.
We say that $f$ satisfies a $(\mu, \theta)$-H\"olderian Error Bound when there exists $\theta\in[0,1/2]$ such that
\BEQ\label{eq:HEB_def}\tag{HEB}
||x - x^*|| \leq \mu (f(x) - f(x^*))^{\theta}.
\EEQ
When the function $f$ is subanalytic, \eqref{eq:HEB_def} is known to hold generically \citep{law1965ensembles,kurdyka1998gradients,bolte2007lojasiewicz}. 
For instance, when $f$ is $(\mu, r)$-uniformly convex with $r\geq 2$ (see Definition~\ref{def:fct_uniform_convexity}), then it satisfies a $((2/\mu)^{1/r}, 1/r)$-H\"olderian Error Bound, which follows from
\[
f(x_t) \geq f(x^*) + \underbrace{\langle \nabla f(x^*) ; x_t - x^* \rangle}_{\geq 0} + \frac{\mu}{2} ||x_t - x^*||_2^r.
\]

Hence we generalize the convergence result of \citep{garber2015faster} and show that as soon as the set $\mathcal{C}$ is $(\alpha, q)$-uniformly convex with $q\geq 2$ and the function $f$ satisfies a non-trivial $(\mu,\theta)$-HEB, the Frank-Wolfe algorithm (with short step) enjoys an accelerated convergence rate with respect to $\mathcal{O}(1/T)$. In particular when $f$ is $\mu$-strongly convex, it satisfies a $(\mu,1/2)$-HEB and by varying $q\geq 2$ we interpolate all sublinear convergence rates between $\mathcal{O}(1/T)$ and $\mathcal{O}(1/T^2)$.

In Lemma~\ref{lemma:scaling_UC_HEB}, we will show an upper bound on $||x_t - v_t||$ when combining the uniform convexity of $\mathcal{C}$ and a H\"olderian Error Bound for $f$. 
Lemma \ref{lemma:scaling_UC_HEB} is then the basis for the convergence analysis and similar to Lemma \ref{lemma:global_scaling_UC}. Overall, Theorem~\ref{th:uniformly_cvx_set_rates_general}, Theorem~\ref{th:uniformly_cvx_set_rates_general_local} and Theorem~\ref{th:rates_uniform_convexity} give an almost complete picture of all the accelerated convergence regimes one can expect with the vanilla Frank-Wolfe algorithm.

\begin{lemma}\label{lemma:scaling_UC_HEB}
Consider a compact and $(\alpha, q)$-uniformly convex set $\mathcal{C}$ with respect to $||\cdot||$. Denote $f$ a strictly convex $L$-smooth function and $x^*=\text{argmin}_{x\in\mathcal{C}}f(x)$. Assume that $f$ satisfies a $(\mu, \theta)$-H\"olderian Error Bound $||x - x^*|| \leq \mu (f(x) - f(x^*))^{\theta}$ with $\theta\in[0,1/2]$. Then for $x_t\in\mathcal{C}$ we have $\alpha/\mu ||x_t - v_t||^{q} h_t^{1-\theta} \leq g(x_t)$, where $g(x_t)$ is the Frank-Wolfe gap and $v_t$ the Frank-Wolfe vertex.
\end{lemma}
\begin{proof}
By Lemma~\ref{lemma:global_scaling_UC} we have $g(x_t) \geq \alpha ||x_t - v_t||^q ||\nabla f(x_t)||_*$.
Then, by combining the convexity of $f$, Cauchy-Schwartz and $(\mu, \theta)$-H\"olderian Error Bound, we have
\[
f(x) - f(x^*) \leq \langle \nabla f(x); x - x^* \rangle \leq ||\nabla f(x)||_* \cdot ||x- x^*|| \leq \mu ||\nabla f(x)||_* \cdot \big( f(x) - f(x^*) \big)^{\theta},
\]
so that $\big(f(x) - f(x^*)\big)^{1-\theta} \leq ||\nabla f(x)||_*$ and finally $g(x_t) \geq \alpha ||x_t - v_t||^q h_t^{1-\theta}$.
\end{proof}

\begin{theorem}\label{th:rates_uniform_convexity}
Consider a $L$-smooth convex function $f$ that satisfies a $(\mu, \theta)$-HEB with $\mu>0$ and $\theta\in]0,1/2]$.
Assume $\mathcal{C}$ is a compact and $(\alpha, q)$-uniformly convex set with respect to $||\cdot||$ with $q \geq 2$.
Then the iterates of the Frank-Wolfe algorithm, with short step or exact line search, satisfy
\BEQ\label{eq:cv_rate_hazan}
f(x_T) - f(x^*) \leq M/(T + k)^{1/(1 - 2\theta/q)},
\EEQ
with $k \triangleq (2 - 2^{\eta})/({2^\eta -1})$ and $ M \triangleq \text{max}\{h_0 k^{1/\eta}, 2/((\eta-(1-\eta)(2^\eta -1))C)^{1/\eta}\}$, where $\eta \triangleq 1 - 2\theta/q$ and $C\triangleq (\alpha/\mu)^{2/q}/L$. In particular for $q=2$ and $\theta=1/2$, we obtain the $\mathcal{O}(1/T^2)$ of \citep{garber2015faster}.
\end{theorem}
\begin{proof}
From the proof of Theorem~\ref{th:uniformly_cvx_set_rates_general}, $L$-smoothness and the step size decision we have
\[
h(x_{t+1}) \leq h(x_t) - \frac{g(x_t)}{2} \cdot \text{min}\Big\{ 1; \frac{g(x_t)}{L ||x_t - v_t||^2}\Big\}.
\]
Then using Lemma \ref{lemma:scaling_UC_HEB}, we can rewrite
\[
\frac{g(x_t)}{||x_t - v_t||^2} = \Big( \frac{g(x_t)^{q/2-1} g(x_t)}{||x_t - v_t||^q}\Big)^{2/q} \geq \Big(\alpha/\mu\Big)^{2/q} g(x_t)^{1 - 2/q}h_t^{(1-\theta)2/q}.
\]
And because $g(x_t) \geq h_t$, we have
\[
\frac{g(x_t)}{||x_t - v_t||^2} \geq \Big(\alpha/\mu\Big)^{2/q} h_t^{1 - 2\theta/q}.
\]
We finally end up with the following recursion
\[
h(x_{t+1}) \leq h(x_t) \cdot \text{max}\Big\{ \frac{1}{2}; 1 - \Big(\alpha/\mu\Big)^{2/q} h_t^{1 - 2\theta/q}/L\Big\},
\]
and we conclude with Lemma~\ref{lem:general_recursive_equation}.
\end{proof}

Overall, Theorem~\ref{th:uniformly_cvx_set_rates_general}, Theorem~\ref{th:uniformly_cvx_set_rates_general_local} and Theorem~\ref{th:rates_uniform_convexity} give an (almost) complete picture of all the accelerated convergence regimes one can expect with the vanilla Frank-Wolfe algorithm.

\section{Online Learning with Linear Oracles and Uniform Convexity}\label{sec:online_optimization}
In online convex optimization, the algorithm sequentially decides an action, a point $x_t$ in a set $\mathcal{C}$, and then incurs a (convex smooth) loss $l_t(x_t)$. 
Algorithms are designed to reduce the cumulative incurred losses over time, $F_t = \frac{1}{t} \sum_{\tau=1}^t l_{\tau}(x_{\tau})$. The comparison to the best action in hindsight is then defined as the \textit{regret} of the algorithm, \textit{i.e.} $R_T \triangleq \sum_{t=1}^{T} l_t(x_t) - \text{min}_{x\in\mathcal{C}} \sum_{t=1}^{T}{l_t(x)}$.

Interesting correspondences have been established between the Frank-Wolfe algorithm and online learning algorithms. 
For instance, recent works \citep{abernethy2017frank,abernethy2018faster} derive new Frank-Wolfe-like algorithms and analyses via two online learning algorithms playing against each other.
Furthermore, a series of work proposed projection-free online algorithms inspired by their offline counterpart, \textit{e.g.} \cite{hazan2012projection} design a Frank-Wolfe online algorithm. 
In following works, \cite{garber2013playing,garber2013linearly} propose projection-free algorithms for online and offline optimization with optimal convergence guarantees where the decision sets are polytopes and the loss functions are strongly-convex. In the same setting, \cite{lafond2015online} analyze the online equivalent of the away-step Frank-Wolfe algorithm via a similar analysis to \citep{lacoste2013affine,lacoste2015global} in the offline setting. Recently, \cite{Hazan2020} proposed a randomized projection-free algorithm that has a regret of $\mathcal{O}(T^{2/3})$ with high probability improving over the deterministic $\mathcal{O}(T^{3/4})$ of \citep{hazan2012projection} and \cite{levy2019projection} designed a projection-free online algorithm over smooth decision sets; dual to uniformly convex sets \citep{vial1983strong}.

\paragraph{Online Linear Optimization and Set Curvature.}
At a high level, when the constraint set is strongly-convex, the analyses of the simple Follow-The-Leader (FTL) for online linear optimization \citep{huang2016following} is analogous to the offline convergence analyses of the Frank-Wolfe algorithm when not assuming strong-convexity of the objective function as in \citep{polyak1966existence,demyanov1970,dunn1979rates}. Indeed, by definition, linear functions do not enjoy non-linear lower bounds, \textit{i.e.} uniform convexity-like assumptions.

In the online linear setting, we write the functions $l_t(x) = \langle c_t; x\rangle$ and assume that $(c_t)$ belong to a bounded set $\mathcal{W}$ (smoothness). FTL consists in choosing the action $x_t$ at time $t$ that minimizes the cumulative sum of the previously observed losses, \textit{i.e.} each iteration solves the minimization of a linear function over $\mathcal{C}$
\BEQ\label{eq:follow_the_leader}
x_T \in \underset{x\in\mathcal{C}}{\text{argmin}} \sum_{t=1}^{T-1} l_t(x) =  \langle \sum_{t=1}^{T-1} c_t; x \rangle.
\EEQ

In general, FTL incurs a worst-case regret of $\mathcal{O}(T)$ \citep{shalev2012online}. For online linear learning, \cite{huang2016following,huang2017following} study the conditions under which the strong convexity of the decision set $\mathcal{C}$ leads to improved regret bounds.
In particular, when there exists a $c>0$ such that for all $T$, $\text{min}_{1\leq t \leq T}||\frac{1}{t}\sum_{\tau=1}^{t}{c_{\tau}}||_* \geq c > 0$, then FTL enjoys the optimal regret bound of $\mathcal{O}(\log(T))$ \citep{huang2017following}. This result is the counter part of the offline geometrical convergence analyses of the Frank-Wolfe algorithm when $\text{inf}_{x\in\mathcal{C}}||\nabla f(x)||_* \geq c >0$ and $\mathcal{C}$ is a strongly convex set \citep{polyak1966existence,demyanov1970,dunn1979rates}.
In Theorem~\ref{thm:online_linear_UC}, we hence further support this analogy between online and offline settings. We show that FTL enjoys continuously interpolated regret bounds between $\mathcal{O}(\log(T))$ and $\mathcal{O}(T)$ for all types of uniform convexity of the decision sets.
Again, this covers a much broader spectrum of \textit{curved} sets, and is similar to Theorem~\ref{th:uniformly_cvx_set_rates_general} in the Frank-Wolfe setting. A proof is deferred to Appendix \ref{app:online}. 

\begin{theorem}\label{thm:online_linear_UC}
Let $\mathcal{C}$ be a compact and $(\alpha, q)$-uniformly convex set with respect to $||\cdot||$. Assume that $L_{T} = \text{min}_{1\leq t \leq T} ||\frac{1}{t}\sum_{\tau=1}^t c_\tau||_* > 0$. Then the regret $R_T$ of FTL \eqref{eq:follow_the_leader} for online linear optimization satisfies
\begin{equation}
  \left\{
    \begin{split}
    R_T &\leq 2M \Big(\frac{2M}{\alpha L_T}\Big)^{1/(q-1)} \Big(\frac{q-1}{q-2}\Big) T^{1 - 1/(q-1)} &~~\text{ when }  q > 2 \\
    R_T &\leq \frac{4M^2}{\alpha L_T} (1 + \log(T)) &~~\text{ when } q=2,
    \end{split}
  \right.
\end{equation}
where $M = \text{sup}_{c\in\mathcal{W}}||c||_*$, with the losses $l_t(x) = \langle c_t; x\rangle$ and $(c_t)$ belong to the bounded set $\mathcal{W}$.
\end{theorem}

The following is the generalization of \citep[(6)]{huang2017following} when the set is uniformly convex (see Definition~\ref{def:set_uniform_convexity}). Note that in our version $\mathcal{C}$ can be uniformly convex with respect to any norm. The proof is deferred to Appendix \ref{app:online}.

\begin{lemma}\label{lem:scaling_UC_online}
Assume $\mathcal{C}\subset\mathbb{R}^d$ is a $(\alpha, q)$-uniformly convex set with respect to $||\cdot||$, with $\alpha > 0$ and $q\geq 2$. Consider the non-zero vectors $\phi_1, \phi_2\in\mathbb{R}^d$ and $v_{\phi_1}\in\argmax_{v\in\mathcal{C}}{\langle \phi; v \rangle}$ and $v_{\phi_2}\in\argmax_{v\in\mathcal{C}}{\langle \phi; v \rangle}$. Then
\BEQ\label{eq:scaling_UC_online}
\langle v_{\phi_1} - v_{\phi_2}; \phi_1 \rangle \leq \Big(\frac{1}{\alpha}\Big)^{1/(q-1)} \frac{||\phi_1 -\phi_2||_*^{1+ 1/(q-1)}}{(\text{max}\{||\phi_1||_*, ||\phi_2||_*\})^{1/(q-1)}},
\EEQ
where $||\cdot||_*$ is the dual norm to $||\cdot||$.
\end{lemma}

\begin{proof}[Proof of Theorem \ref{thm:online_linear_UC}]
The proof follows exactly that of \citep[Theorem 5]{huang2017following}.
Write $M = \text{sup}_{c\in\mathcal{F}}||c||$, $F_t(x)= \frac{1}{t}\sum_{\tau=1}^t{\langle c_t; x\rangle}$ and short cut $\nabla F_t\triangleq \frac{1}{t}\sum_{\tau=1}^t{c_t}$ the gradient of the linear function $F_t(x)$. Recall that with FTL, $x_t$ is defined as
\[
x_t \in \text{argmin}_{x\in\mathcal{C}}  \langle \sum_{\tau=1}^{t-1} c_t; x \rangle.
\]
As in \citep[Theorem 5]{huang2017following} we have (for any norm $||\cdot||$)
\[
||\nabla F_t - \nabla F_{t-1}|| \leq \frac{2M}{t}.
\]
Using \citep[Proposition 2]{huang2017following} and Lemma~\ref{lem:scaling_UC_online} we get the following upper bound on the regret
\[
R_T = \sum_{t=1}^T t \langle x_{t+1} - x_t; \nabla F_t \rangle \leq \Big(\frac{1}{\alpha}\Big)^{1/(q-1)} \sum_{t=1}^T t \frac{||\nabla F_t -\nabla F_{t-1}||_*^{1+ 1/(q-1)}}{(\text{max}\{||\nabla F_t||_*, ||\nabla F_{t-1}||_*\})^{1/(q-1)}}.
\]
Hence, with $L_T = \text{min}_{1\leq t \leq T} ||\nabla F_t||_* > 0$, we have
\[
R_T \leq 2M\Big(\frac{2M}{\alpha L_T}\Big)^{1/(q-1)} \sum_{t=1}^T t^{-1/(q-1)}.
\]
Then we have for $q > 2$
\[
\sum_{t=1}^T t^{-1/(q-1)} = 1 + \sum_{t=2}^T t^{-1/(q-1)} \leq 1 + \int_{x=1}^{T-1}x^{- 1/(q-1)}dx = 1 + \Big[ \frac{t^{1 - 1/(q-1)}}{1-1/(q-1)}\Big]_{1}^{T-1},
\]
so that finally
\[
R_T \leq 2M\Big(\frac{2M}{\alpha L_T}\Big)^{1/(q-1)} \Big(\frac{q-1}{q-2}\Big) T^{1 - 1/(q-1)}.
\]
\end{proof}

With the simple FTL, we obtain non-trivial regret bounds, \textit{i.e.} $o(T)$,  whenever the set is uniformly convex, without any curvature assumption on the loss functions (because they are linear). In particular for $q\in[2,3]$, it improves over the general tight regret bound of $\mathcal{O}(\sqrt{T})$ for smooth convex losses and compact convex decision sets \citep{shalev2012online}. Interestingly, with the same assumption on $\mathcal{C}$, \cite{dekel2017online} obtain for online linear optimization, the same asymptotical regret bounds with a variation of Follow-The-Leader incorporating hints. It is remarkable that the presence of hints or the assumption $\text{min}_{1\leq t \leq T} ||\frac{1}{t}\sum_{\tau=1}^t c_\tau||_* \geq c > 0$ for all $T$ both lead to the same bounds.

\section{Examples of Uniformly Convex Objects}\label{sec:example_uniform_convexity}
The uniform convexity assumptions refine the convex properties of several mathematical objects, such as normed spaces, functions, and sets. 
In this section, we provide some connection between these various notions of uniform convexity. 
In Section \ref{ssec:uniform_convexity_space}, we recall that norm balls of uniformly convex spaces are uniformly convex sets, and show set uniform convexity of classic norm balls in Section \ref{ssec:uniform_convexity_classic_balls} and illustrate it with numerical experiments in Section \ref{sec:numerical_examples}. In Appendix \ref{ssec:uniform_convexity_functions}, we show that the level sets of some uniformly convex functions are uniformly convex sets, extending the strong convexity results of \citep[Section 5]{garber2015faster}.

\subsection{Uniformly Convex Spaces}\label{ssec:uniform_convexity_space}

The uniform convexity of norm balls (Definition \ref{def:set_uniform_convexity}) is closely related to the uniform convexity of normed spaces \citep{polyak1966existence,balashov2011uniformly,lindenstrauss2013classical,weber2013local}. 
Some works establish sharp uniform convexity results for classical normed spaces such as $l_p$, $L_p$ or $C_p$. 
Most of the practical examples of uniformly convex sets are norm balls and are hence tightly linked with uniformly convex spaces. 
The property of these sets has many consequences, \textit{e.g.} \citep{donahue1997rates}. It also relates to concentration inequalities in Banach Spaces \citep{juditsky2008large} and hence implications \citep{ivanov2019approximate} for approximate versions of the Carath\'eodory theorem \citep{combettes2019revisiting}.

\citep{clarkson1936uniformly,boas1940some} define a uniformly convex normed space $(\mathbb{X}, ||\cdot||)$ as a normed space such that, for each $\epsilon > 0$, there is a $\delta >0$ such that if $x$ and $y$ are unit vectors in $\mathbb{X}$ with $||x - y||\geq \epsilon$, then $(x+y)/2$ has norm lesser or equal to $1 - \delta$. Specific quantification of spaces satisfying this property is obtained via the modulus of convexity, a measure of non-linearity of a norm.

\begin{definition}[Modulus of convexity]\label{def:modulus_convexity}
The modulus of convexity of the space $(\mathbb{X}, ||\cdot||)$ is defined as
\BEQ\label{eq:modulus_convexity}
\delta_{\mathbb{X}}(\epsilon) = \text{inf }\Big\{ 1 - \Big|\Big|\frac{x+y}{2}\Big|\Big|~\Big|~||x||\leq  1~,~||y|| \leq 1~,~||x-y|| \geq \epsilon \Big\}.
\EEQ
\end{definition}

A normed space $\mathbb{X}$ is said to be $r$-uniformly convex in the case $\delta_{\mathbb{X}}(\epsilon)\geq C \epsilon^r$.
These specific lower bounds on the modulus of convexity imply that the balls stemming for such spaces are uniformly convex in the sense of Definition \ref{def:set_uniform_convexity}.
There exist sharp results for $L_p$ and $\ell_p$ spaces in \citep{clarkson1936uniformly,hanner1956}.
Matrix spaces with $p$-Schatten norm are known as $C_p$ spaces, and sharp results concerning their uniform convexity can be found in \citep{dixmier1953formes,tomczak1974moduli,simon1979trace,ball1994sharp}.
The following gives a link between the set $\gamma_{\mathcal{C}}$ and space $\delta_{\mathbb{X}}$ modulus of convexity, see proof in Appendix~\ref{ssec:uniformly_convex_spaces}. 

\begin{lemma}\label{lem:link_space_set_UC}
If a normed space $(\mathbb{X}, ||\cdot||)$ is uniformly convex with modulus of convexity $\delta_{\mathbb{X}}(\cdot)$, then its unit norm ball is $\delta_{\mathbb{X}}(\cdot)$ uniformly convex with respect to $||\cdot||$. Note that if the unit ball $B_{||\cdot||}(1)$ is $(\alpha, q)$-uniformly convex, then $B_{||\cdot||}(r)$ is $({\alpha}/{r^{q-1}}, q)$-uniformly convex.
\end{lemma}

\subsection{Uniform Convexity of Some Classic Norm Balls}\label{ssec:uniform_convexity_classic_balls}

When $p\in]1,2]$, $\ell_p$-balls are strongly convex sets and $((p-1)/2, 2)$-uniformly convex with respect to $||\cdot||_p$, see for instance \citep[Theorem 2]{hanner1956} or \citep[Lemma 4]{garber2015faster}. When $p>2$, the $\ell_p$-balls are $(1/p, p)$-uniformly convex with respect to $||\cdot||_p$ \citep[Theorem 2]{hanner1956}. Uniform convexity also extends the strong convexity of group $\ell_{s,p}$-norms (with $1 < p,s \leq 2$) \citep[\S 5.3. and 5.4.]{garber2015faster} to the general case $p,s > 1$.

\citep{dixmier1953formes,tomczak1974moduli,simon1979trace,ball1994sharp} focus of the uniform convexity of the $(C_p, ||\cdot||_{S(p)})$ spaces, \textit{i.e.} spaces of matrix where the norm is the $\ell_p$-norm of a matrix singular values . Their unit balls are hence the $p$-Schatten balls. For $p\in]1,2]$, $p$-Schatten balls are $((p-1)/2, 2)$-uniformly convex with respect to $||\cdot||_{S(p)}$, see \citep[Lemma 6]{garber2015faster} and the sharp results of \citep{ball1994sharp}. For the case $p > 2$, \citep{dixmier1953formes} showed that the $p$-Schatten balls are $(1/p, p)$-uniformly convex with respect to $||\cdot||_{S(p)}$, see also \citep[\S III]{ball1994sharp}.

\section{Numerical Illustration}\label{sec:numerical_examples}
Uniform convexity is a global assumption. Hence, in Theorem \ref{th:uniformly_cvx_set_rates_general}, we obtain sublinear convergence that do not depend on the specific location of the solution $x^*\in\partial\mathcal{C}$. 
However, some regions of $\mathcal{C}$ might be relatively more curved than others and hence exhibit faster convergence rates. This effect is quantified in Theorem~\ref{th:uniformly_cvx_set_rates_general_local} when a local scaling inequality holds.

In Figure \ref{fig:exp_quadratic_lp}, in the case of the $\ell_p$-balls with $p>2$, we vary the approximate location of the optimum $x^*$ in the boundary of the $\ell_p$-balls. 

Subfigures \eqref{fig:opt_curved_deter}, \eqref{fig:opt_curved_short_step}, and \eqref{fig:opt_curved_exact} are associated to an optimization problem where the solution $x^*$ of \eqref{eq:opt_problem} is near the intersection of the $\ell_p$-balls and the half-line generated by $\sum_{i=1}^{d}e_i$ (where the $(e_i)$ is the canonical basis), \textit{i.e.} in \textit{curved} regions of the boundaries of the $\ell_p$-balls.

Subfigures \eqref{fig:opt_flat_deter}, \eqref{fig:opt_flat_short_step}, and \eqref{fig:opt_flat_exact} corresponds to the same optimization problem where the solution $x^*$ to \eqref{eq:opt_problem} is close to the intersection between the half-line generated by $e_1$ and the boundary of the $\ell_p$-balls, \textit{i.e.} in \textit{flat} regions of the boundaries of the $\ell_p$-balls.

We observe that when the optimum is at a \textit{curved} location, the convergence is quickly linear for $p$ sufficiently close to $2$ and appropriate line-search (see Subfigures \eqref{fig:opt_curved_short_step} and \eqref{fig:opt_curved_exact}).
However, when the optimum is near the \textit{flat} location, we indeed observe sublinear convergence rates (see Subfigures \eqref{fig:opt_flat_short_step} and \eqref{fig:opt_flat_exact}). 
It still becomes linear for $p=2.1$ with exact line-search in Subfigure \eqref{fig:opt_flat_exact}.

Also, Theorem \ref{th:uniformly_cvx_set_rates_general} gives accelerated rates when using the Frank-Wolfe algorithm with exact line-search or short step. In Subfigures \eqref{fig:opt_curved_deter} and \eqref{fig:opt_flat_deter}, we show examples of the convergence of the Frank-Wolfe algorithm when using deterministic line-search. The rates are indeed sublinear in $\mathcal{O}(1/T)$. In other words, deterministic line-search generally do not lead to accelerated convergence rates when the sets are uniformly convex. 

\begin{figure}[h]
\subfloat[\label{fig:opt_curved_deter}]{{\includegraphics[width=0.32\linewidth]{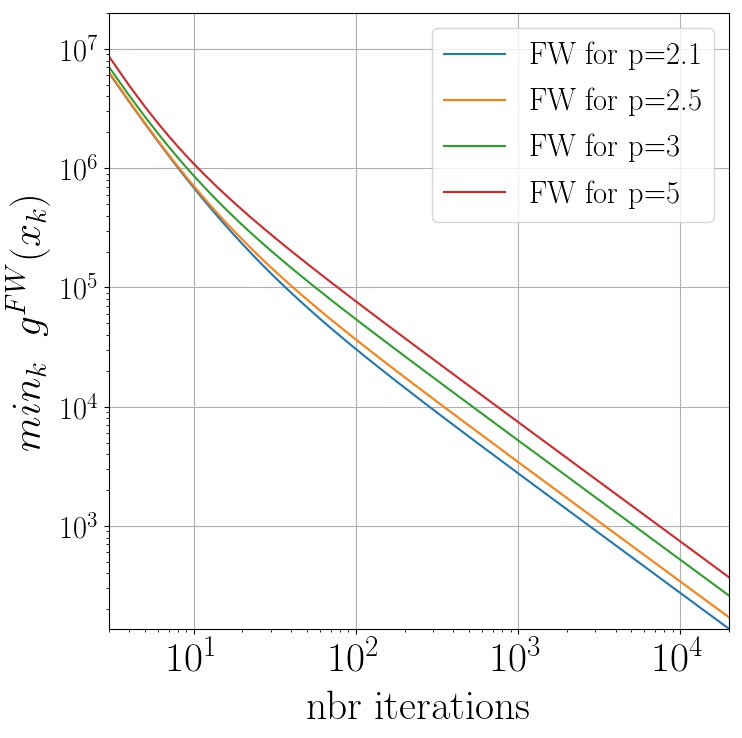} }}
\subfloat[\label{fig:opt_curved_short_step}]{{\includegraphics[width=0.32\linewidth]{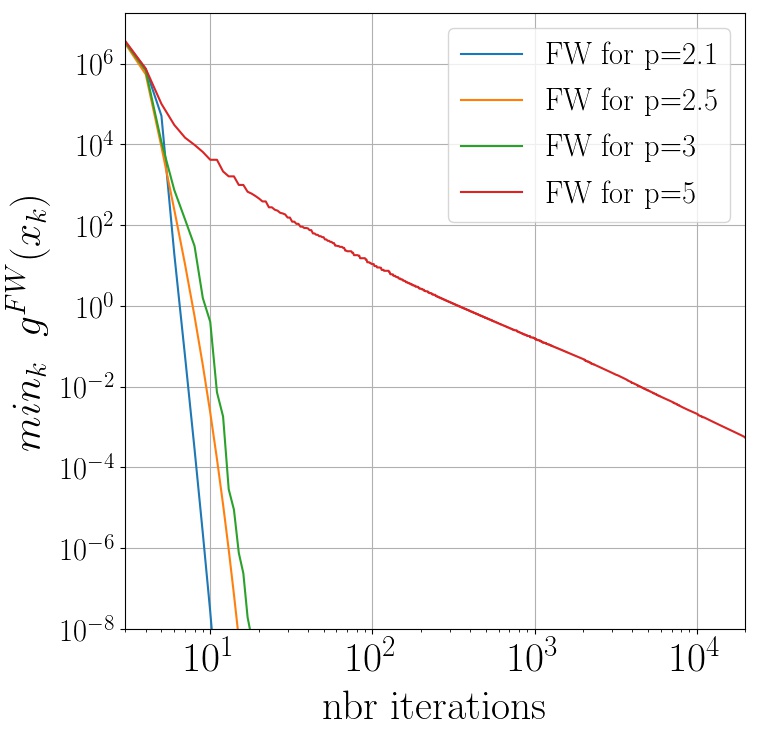} }}
\subfloat[\label{fig:opt_curved_exact}]{{\includegraphics[width=0.32\linewidth]{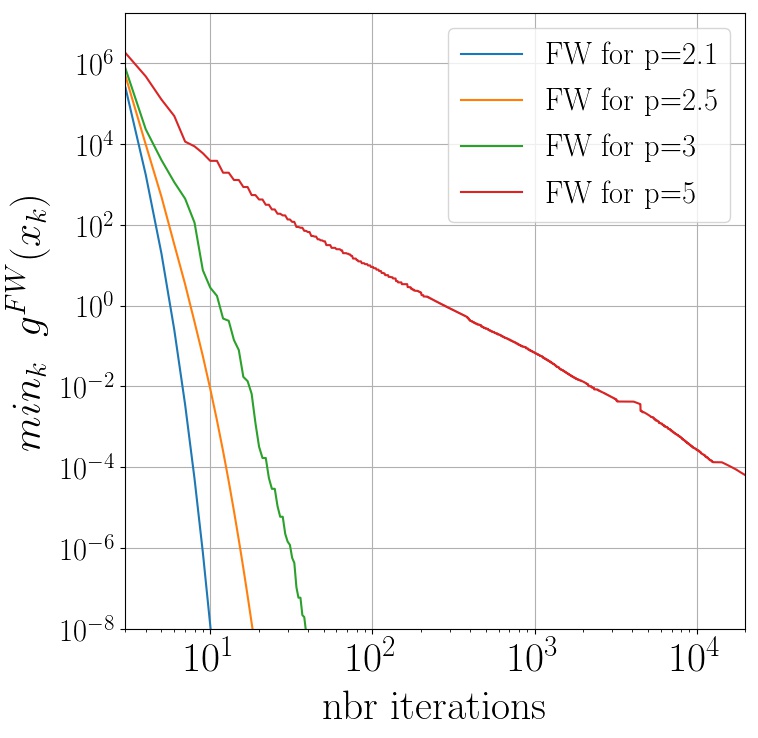} }}

\subfloat[\label{fig:opt_flat_deter}]{{\includegraphics[width=0.32\linewidth]{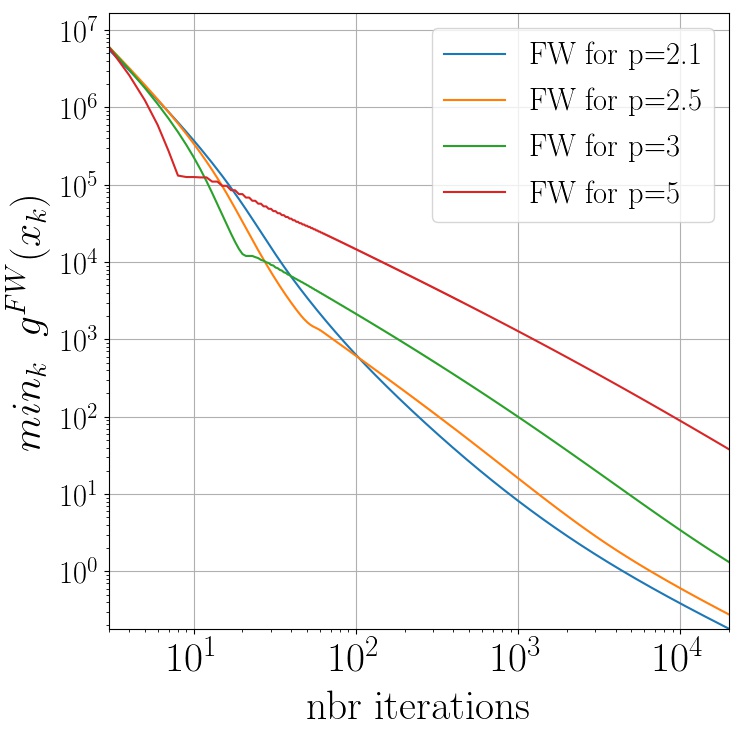} }}
\subfloat[\label{fig:opt_flat_short_step}]{{\includegraphics[width=0.32\linewidth]{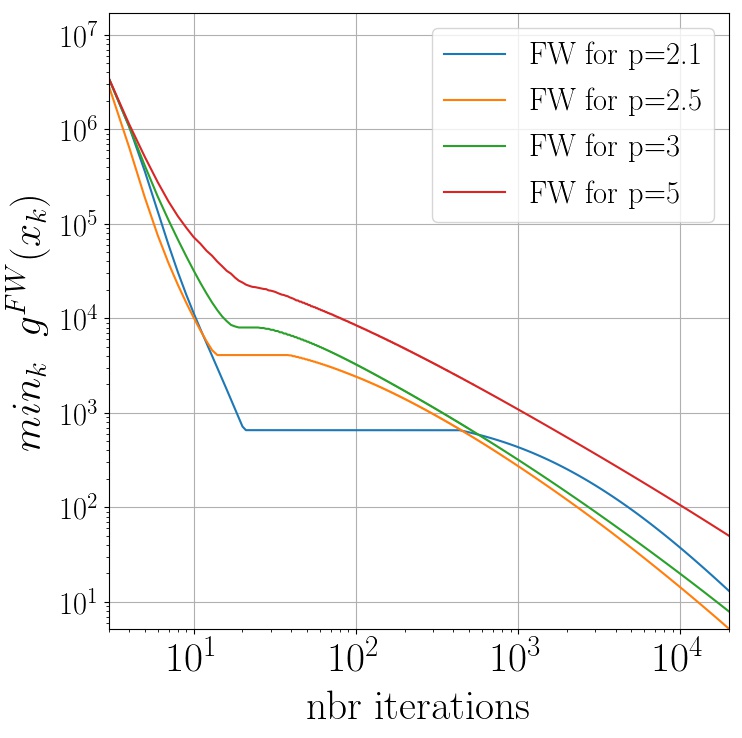}}}
\subfloat[\label{fig:opt_flat_exact}]{{\includegraphics[width=0.32\linewidth]{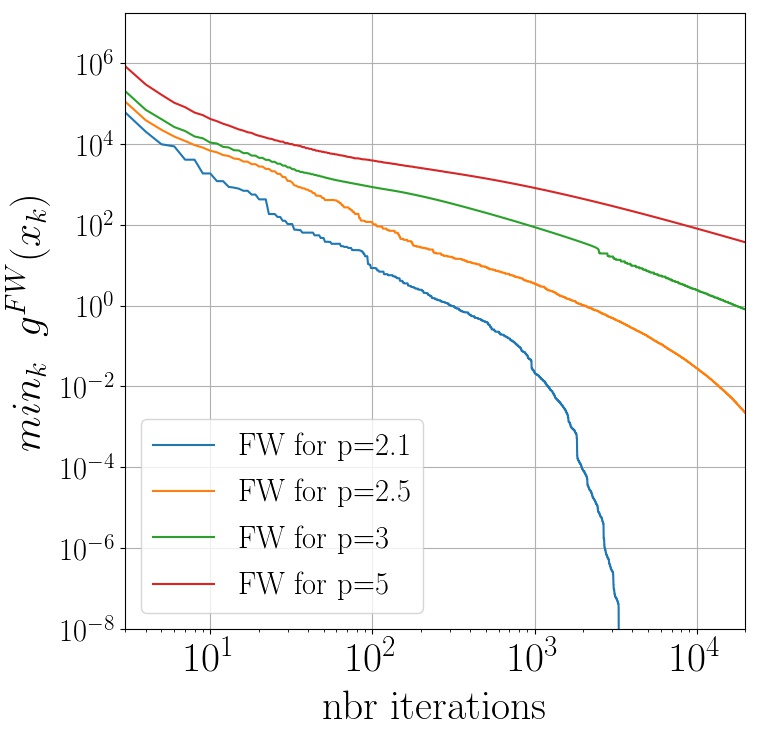} }}
\caption{Solving \eqref{eq:opt_problem} with the Frank-Wolfe algorithm where $f$ is a quadratic with condition number $100$ and the constraint sets are various $\ell_p$-balls of radius $5$. 
We vary $p$ so that all balls are uniformly convex but not strongly-convex. 
We vary the position of the solution to \eqref{eq:opt_problem} with respect to the boundaries of the constraints sets. 
On the first row, we choose the constrained optimum close to the intersection of the set boundary and the line generated by $\sum_i e_i$ (where the $e_i$ form the canonical basis), where $\ell_p$-balls are typically more \textit{curved}. 
On the second row, we choose the constrained optimum near the intersection between the set boundary and the line generate by $e_1$, a region where the $\ell_p$-balls are \textit{flat}. 
On a line, each plot exhibits the behavior of the Frank-Wolfe algorithm iterates with different step size strategy: deterministic line-search (i.e. $1/(k+1)$), short step and exact line-search.
To avoid the oscillating behavior of Frank-Wolfe gap, the $y$-axis represents $\min_{k=1,\ldots,T}{g(x_k)}$ where $g(\cdot)$ is the Frank-Wolfe gap and $T$ the number of iterations.}\label{fig:exp_quadratic_lp}
\end{figure}

\section{Conclusion}
Our results fill the gap between known convergence rates for the Frank Wolfe algorithm. Qualitatively, they also mean that it is the \textit{curvature} of the constraint set that accelerates the convergence of the Frank-Wolfe algorithm, not just strong-convexity. This emphasis on curvature echoes works in other settings \citep{huang2016following}. For the sake of theory, the results could be immediately refined by measuring the \textit{local} curvature of convex bodies with more sophisticated tools than uniform convexity \citep{schneider2015curvatures}.

From a more practical perspective, uniform convexity encompasses ubiquitous structures of constraint sets appearing in machine learning and signal processing. In applications where the (e.g. regularization) constraints are likely to be active, the assumption that $\text{inf}_{x\in\mathcal{C}}||\nabla f(x)||_* > 0$ is not restrictive and the value of $c$ quantifies the relevance of the constraints.


\section*{Acknowledgements}
TK would like to thanks Pierre-Cyril Aubin for very interesting discussions on Banach Spaces, which contributed is the motivation for studying the convergence rates of projection-free methods with uniform convexity assumptions. TK acknowledges funding from the CFM-ENS chaire {\em les mod\`eles et sciences des donn\'ees}. AA is at CNRS \& d\'epartement d'informatique, \'Ecole normale sup\'erieure, UMR CNRS 8548, 45 rue d'Ulm 75005 Paris, France,  INRIA  and  PSL  Research  University. AA acknowledges support from the ML  \& Optimisation joint research initiative with the fonds AXA pour la recherche and Kamet Ventures, as well as a Google focused award.


\newpage

\appendix

\section{Recursive Lemma}\label{app:technical_lemma}
The proofs of Theorems~\ref{th:uniformly_cvx_set_rates_general}, \ref{th:uniformly_cvx_set_rates_general_local}, and \ref{th:rates_uniform_convexity} involve finding explicit bounds for sequences $(h_t)$ satisfying recursive inequalities of the form,
\BEQ\label{eq:general_recursive_equation}
h_{t+1} \leq h_t \cdot \text{max}\{1/2, 1 - C h_t^{\eta} \}.
\EEQ
with $\eta<1$. An explicit solution with $\eta=1/2$ is given in \citep{garber2015faster} and corresponds to $h_t=\mathcal{O}(1/T^2)$, while for $\eta=1$ we recover the classical sublinear Frank-Wolfe regime of $\mathcal{O}(1/T)$. For a $\eta\in]0,1]$, we have $\mathcal{O}(1/T^{1/\eta})$ (see for instance \citep{temlyakov2011greedy} or \citep[Lemma 4.2.]{nguyen2017greedy}), which can be guessed via $h(t)=(C \eta)^{1/\eta} t^{-1/\eta}$ the solution of the differential equation $h^\prime(t)=-C h(t)^{\eta+1}$ for $t>0$.
A quantitative statement is, for instance, given in \citep[proof of Theorem 1.]{YiAdapativeFW} that we reproduce here. 

\begin{lemma}[Recurrence and sub-linear rates]\label{lem:general_recursive_equation}
Consider a sequence $(h_t)_{t\in\mathbb{N}}$ of non-negative numbers satisfying \eqref{eq:general_recursive_equation} with $0 < \eta \leq 1$, then $h_T = \mathcal{O}\big(1/T^{1/\eta} \big)$. More precisely for all $t\geq 0$,
\[
h_t \leq \frac{M}{(t+k)^{1/\eta}}
\]
with $k \triangleq (2 - 2^{\eta})/({2^\eta -1})$ and $ M \triangleq \text{max}\{h_0 k^{1/\eta}, 2/((\eta-(1-\eta)(2^\eta -1))C)^{1/\eta}\}$.
\end{lemma}

\section{Beyond Local Uniform Convexity}

Here we show that additional convexity properties on the gauge function of $\mathcal{C}$ imply local scaling inequalities on $\mathcal{C}$. Note that for ease, we assume that the gauge function is differential at $x^*$ which is not necessarily the case case when the set $\mathcal{C}$ is uniformly convex.

\begin{lemma}\label{lem:UC_gauge_to_scaling}
Consider a compact convex set $\mathcal{C}$ with $0$ in its interior.
Assume the gauge function of $\mathcal{C}$ is differentiable and normal cone at the boundary are half-lines. Assume $(\mu, r)$-uniformly convex at $x^*$ a solution of \eqref{eq:opt_problem} (where $f$ is a convex $L$-smooth function and $\text{inf}_{x\in\mathcal{C}} ||x||_{\mathcal{C}} > 0$), then we have the following scaling inequality for all $x\in\mathcal{C}$
\[
\langle -\nabla f(x^*); x - x^*\rangle \geq \frac{\mu}{||g||} ||\nabla f(x^*)|| ||x - x^*||^q,
\]
where $g\in N_{\mathcal{C}}(x^*)$ and $\langle g; x^*\rangle=1$.
\end{lemma}
\begin{proof}
We have $x^*\in\partial\mathcal{C}$. Write $g = \nabla ||x||_{\mathcal{C}}$. Then by $(\mu, r)$-uniformly convex of the gauge function we have
\[
||x||_{\mathcal{C}} \geq \underbrace{||x^*||_{\mathcal{C}}}_{=1} + \langle g; x - x^*\rangle + \mu ||x - x^*||^q.
\]
Hence we have
\[
\langle -g; x - x^*\rangle   \geq \underbrace{1 - ||x||_{\mathcal{C}}}_{\geq 0} + \mu ||x - x^*||^q \geq \mu ||x - x^*||^q.
\]
When it is differentiable, \citep[(1.39)]{schneider2014convex} show that $g$ satisfies $g\in N_{\mathcal{C}}(x^*)$ and $\langle g; x^*\rangle = 1$. Here, the normal cone is a half-line and $-\nabla f(x^*)\in N_{\mathcal{C}}(x^*)$. In particular then $-\nabla f(x^*) = \frac{||\nabla f(x^*)||}{||g||} g$. Finally
\[
\langle -\nabla f(x^*); x - x^*\rangle \geq \frac{\mu}{||g||} ||x - x^*||^q ||\nabla f(x^*)||.
\]
\end{proof}

\section{Proofs in Online Optimization}\label{app:online}

The following is the generalization of \citep[(6)]{huang2017following} when the set is uniformly convex. Note that in our version $\mathcal{C}$ can be uniformly convex with respect to any norm.

\begin{lemma}\label{lem:scaling_UC_online}
Assume $\mathcal{C}\subset\mathbb{R}^d$ is a $(\alpha, q)$-uniformly convex set with respect to $||\cdot||$, with $\alpha > 0$ and $q\geq 2$. Consider the non-zero vectors $\phi_1, \phi_2\in\mathbb{R}^d$ and $v_{\phi_1}\in\argmax_{v\in\mathcal{C}}{\langle \phi; v \rangle}$ and $v_{\phi_2}\in\argmax_{v\in\mathcal{C}}{\langle \phi; v \rangle}$. Then
\BEQ\label{eq:scaling_UC_online}
\langle v_{\phi_1} - v_{\phi_2}; \phi_1 \rangle \leq \Big(\frac{1}{\alpha}\Big)^{1/(q-1)} \frac{||\phi_1 -\phi_2||_*^{1+ 1/(q-1)}}{(\text{max}\{||\phi_1||_*, ||\phi_2||_*\})^{1/(q-1)}},
\EEQ
where $||\cdot||_*$ is the dual norm to $||\cdot||$.
\end{lemma}
\begin{proof}
By definition of uniform convexity, for any $z$ of unit norm, $v_{\gamma}(z)\in\mathcal{C}$ where
\[
v_{\gamma}(z) = \gamma v_{\phi_1} + (1 - \gamma) v_{\phi_2} + \gamma (1 - \gamma) \alpha ||v_{\phi_1} - v_{\phi_2}||^q z.
\]
By optimality of $v_{\phi_1}$ and $v_{\phi_2}$, we have $\langle v_{\gamma}(z); \phi_1\rangle \leq \langle v_1; \phi_1\rangle$ and $\langle v_{\gamma}(z); \phi_2\rangle \leq \langle v_2; \phi_2\rangle$, so that
\[
\langle v_{\gamma}(z) ; \gamma \phi_1 + (1-\gamma) \phi_2 \rangle \leq \gamma \langle v_1 ; \phi_1 \rangle + (1 - \gamma ) \langle v_2 ; \phi_2 \rangle.
\]
Write $\phi_{\gamma} = \gamma \phi_1 + (1-\gamma) \phi_2$. Then, when developing the left hand side, we get
\[
\gamma (1 - \gamma) \alpha ||v_{\phi_1} - v_{\phi_2}||^q \langle z; \phi_{\gamma}\rangle \leq \gamma (1 - \gamma) \langle v_{\phi_1} - v_{\phi_2}; \phi_1 - \phi_2\rangle 
\]
Choosing the best $z$ of unit norm we get
\[
\alpha ||v_{\phi_1} - v_{\phi_2}||^q || \phi_{\gamma}||_{*} \leq \langle v_{\phi_1} - v_{\phi_2}; \phi_1 - \phi_2\rangle
\]
and for $\gamma=0$ and $\gamma=1$ and via generalized Cauchy-Schwartz we get
\[
\alpha ||v_{\phi_1} - v_{\phi_2}||^q \cdot\text{max}\{||\phi_1||_*,||\phi_2||_*\} \leq ||v_{\phi_1} - v_{\phi_2}|| \cdot ||\phi_1 - \phi_2||_*.
\]
Then,
\[
\langle v_{\phi_1} - v_{\phi_2}; \phi_1 \rangle \leq ||v_{\phi_1} - v_{\phi_2}|| \cdot ||\phi_1 - \phi_2||_* \leq \Big(\frac{1}{\alpha}\Big)^{1/(q-1)}  \frac{||\phi_1 -\phi_2||_*^{1 +  1/(q-1)}}{(\text{max}\{||\phi_1||_*, ||\phi_2||_*\})^{1/(q-1)}},
\]
and we finally obtain \eqref{eq:scaling_UC_online}.
\end{proof}

\section{Uniformly Convex Objects}

\subsection{Uniformly Convex Spaces}\label{ssec:uniformly_convex_spaces}

\begin{proof}[Proof of Lemma~\ref{lem:link_space_set_UC}]
Assume $(\mathbb{X}, ||\cdot||)$ is uniformly convex with modulus of convexity $\delta(\cdot)$. Then for any $(x,y,z)\in B_{||\cdot||}(1)$, we have by definition $1 - \frac{||x + y||}{2} \geq \delta(||x-y||)$ and then 
\[
\Big|\Big|\frac{x+y}{2} + \delta(||x-y||) z\Big|\Big| \leq \Big|\Big|\frac{x+y}{2}\Big|\Big| + \delta(||x-y||) \leq 1~.
\]
Hence, $\frac{x+y}{2} + \delta(||x-y||) z\in\mathcal{C}$. 
Without loss of generality, consider $\eta\in]0;1/2]$. 
We need to show that $\eta x + (1 - \eta) y + \delta(||x-y||)z\in\mathcal{C}$ for any $z$ with norm lesser than $1$.
First, note that $\eta x + (1 - \eta) y = (1 - 2\eta) y + (2 \eta) (x+y)/2$. Note also that because $1 - 2\eta\in[0,1]$, we have for any $z$ of norm lesser than $1$
\[
(1 - 2\eta) x + (2 \eta) \big[ (x+y)/2 + \delta(||x-y||) z\big] \in \mathcal{C}.
\]
Hence, for any $z$ of norm lesser than $1$, we have
\[
\eta x + (1 - \eta) y + 2 \eta \delta(||x-y||) z \in\mathcal{C}.
\]
Or equivalently
\[
\eta x + (1 - \eta) y + (1-\eta) \eta \delta(||x-y||) \frac{2\eta}{(1-\eta) \eta} z \in\mathcal{C}.
\]
Because $\frac{2\eta}{(1-\eta) \eta} \geq 1$, it follows that for any $z$ of norm lesser than $1$ we have 
\[
\eta x + (1 - \eta) y + (1-\eta) \eta \delta(||x-y||) z \in\mathcal{C},
\]
which conclude on the uniform convexity of the norm ball.
\end{proof}

\subsection{Uniformly Convex Functions}\label{ssec:uniform_convexity_functions}
Uniform convexity is also a property of convex functions and defined as follows.

\begin{definition}\label{def:fct_uniform_convexity}
A differentiable function $f$ is $(\mu, r)$-uniformly convex on a convex set $\mathcal{C}$ if there exists $r\geq 2$ and $\mu>0$ such that for all $(x,y)\in\mathcal{C}$
\[
f(y) \geq f(x) + \langle \nabla f(x); y - x\rangle + \frac{\mu}{2} ||x-y||_2^r~.
\]
\end{definition}
We now state the equivalent of \citep[Theorem 12]{journee2010generalized} for the level sets of uniformly convex functions. This was already used in \citep{garber2015faster} in the case of strongly-convex sets.

\begin{lemma}\label{lem:level_set_uni_smooth}
Let $f:\mathbb{R}^d\rightarrow\mathbb{R}^+$ be a non-negative, $L$-smooth and $(\mu, r)$-uniformly convex function on $\mathbb{R}^d$, with $r\geq 2$. Then for any $w>0$, the set
\[
\mathcal{L}_{w} = \Big\{x~|~ f(x)\leq w \Big\}~,
\]
is $(\alpha, r)$-uniformly convex with $\alpha=\frac{\mu}{\sqrt{2 w L}}$.
\end{lemma}
\begin{proof}
The proof follows exactly that of \citep[Theorem 12]{journee2010generalized}, replacing $||x-y||^2$ with $||x-y||^r$. We state it for the sake of completeness. Consider $w_0>0$, $(x,y)\in \mathcal{L}_{w}$ and $\gamma \in [0,1]$. We denote $z= \gamma x + (1 - \gamma) y$. For $u\in\mathbb{R}^d$, by $L$-smoothness applied at $z$ and at $x^*$ (the unconstrained optimum of $f$), we have
\begin{eqnarray*}
f(z + u) & \leq & f(z) + \langle \nabla f(z); u\rangle + \frac{L}{2} ||u||_2^2\\
         & \leq & f(z) + ||\nabla f(z)||\cdot ||u|| + \frac{L}{2} ||u||_2^2\\
         & \leq & f(z) + \sqrt{2L f(z)}||u|| + \frac{L}{2} ||u||_2^2 = \Big( \sqrt{f(z)} + \sqrt{\frac{L}{2}} ||u|| \Big)^2~.
\end{eqnarray*}
Note that uniform convexity of $f$ implies that 
\[
f(z) \leq \gamma f(x) + (1-\gamma) f(y) - \frac{\mu}{2} \gamma (1 - \gamma) ||x-y||^r 
\]
In particular then, because $x,y\in \mathcal{L}_w$,  we have $f(z)\leq w - \frac{\mu}{2} \gamma (1 - \gamma) ||x-y||^r$ so that
\BEQ
f(z + u) \leq \Big( \sqrt{w - \frac{\mu}{2} \gamma (1 - \gamma) ||x-y||^r} + \sqrt{\frac{L}{2}} ||u|| \Big)^2
\EEQ
Leveraging on the concavity of the square-root, we get
\BEQ
f(z + u) \leq \Big( \sqrt{w} - \frac{\mu}{4\sqrt{w}} \gamma (1 - \gamma) ||x-y||^r + \sqrt{\frac{L}{2}} ||u|| \Big)^2~.
\EEQ
Hence for any $u$ such that
\[
||u|| = \frac{\mu}{2\sqrt{2 w L}} \gamma (1 - \gamma) ||x-y||^r~,
\]
we have $z + u\in\mathcal{L}_w$. Hence $\mathcal{L}_w$ is a $(\frac{\mu}{2\sqrt{2 w L}}, r)$-uniformly convex set.
\end{proof}

Lemma \ref{lem:level_set_uni_smooth} restrictively requires smoothness of the uniformly convex function $f$. Hence we provide the analogous of \citep[Lemma 3]{garber2015faster}.

\begin{lemma}\label{lem:level_set_uni}
Consider a finite dimensional normed vector space $(\mathbb{X}, ||\cdot||)$. Assume $f(x)=||x||^2$ is $(\mu, s)$-uniformly convex function (with $r\geq 2$) with respect to $||\cdot||$. Then the norm balls $B_{||\cdot||}(r)=\Big\{x\in\mathbb{X}~|~ ||x||\leq r\Big\}$ are $(\frac{\mu}{2r}, s)$-uniformly convex.
\end{lemma}
\begin{proof}
The proof follows exactly that of \citep[Lemma 3]{garber2015faster} which itself follows that of \citep[Theorem 12]{journee2010generalized}, where operations involving $L$-smoothness are replaced by an application of the triangular inequality.

Let's consider $s \geq 2$, $(x,y)\in B_{||\cdot||}(r)$ and $\gamma \in [0,1]$. We denote $z= \gamma x + (1 - \gamma) y$. For $u\in\mathbb{X}$, applying successively triangular inequality and $(\mu, s)$-uniform convexity of $f(x)=||x||^2$, we get
\begin{eqnarray*}
f(z+u) = ||z + u||^2 &\leq& \Big(\sqrt{f(z)} + ||u||\Big)^2 \\
&\leq& \Big(\sqrt{r^2 - \frac{\mu}{2}\gamma(1-\gamma)||x-y||^s} + ||u||\Big)^2~.
\end{eqnarray*}
We then use concavity of the square root as before to get
\[
||z + u||^2 \leq \Big(r - \frac{\mu}{4 r}\gamma(1-\gamma)||x-y||^s + ||u||\Big)^2~.
\]
In particuler, for $u\in\mathbb{X}$ such that $||u|| = \frac{\mu}{4 r}\gamma(1-\gamma)||x-y||^s$, we have $z+u\in B_{||\cdot||}(r)$. Hence $B_{||\cdot||}(r)$ is $(\frac{\mu}{2r},s)$- uniformly convex with respect to $||\cdot||$.
\end{proof}

These previous lemmas hence allow to translate functional uniformly convex results into results for classic balls norms. For instance, \citep[Lemma 17]{shai2007phd} showed that for $p\in]1,2]$ $f(x) = 1/2||x||_p^2$ was $(p-1)$-uniformly convex with respect to $||\cdot||_p$.

\end{document}

%% file: defs.tex
\newtheorem{theorem}{Theorem}[section]

\newtheorem{definition}[theorem]{Definition}
\newtheorem{remark}[theorem]{Remark}
\newtheorem{lemma}[theorem]{Lemma}
\newtheorem{corollary}[theorem]{Corollary}

\newcommand{\BEAS}{\begin{eqnarray*}}
\newcommand{\EEAS}{\end{eqnarray*}}
\newcommand{\BEA}{\begin{eqnarray}}
\newcommand{\EEA}{\end{eqnarray}}
\newcommand{\BEQ}{\begin{equation}}
\newcommand{\EEQ}{\end{equation}}
\newcommand{\BIT}{\begin{itemize}}
\newcommand{\EIT}{\end{itemize}}
\newcommand{\BNUM}{\begin{enumerate}}
\newcommand{\ENUM}{\end{enumerate}}

\newcommand{\BA}{\begin{array}}
\newcommand{\EA}{\end{array}}

\newcommand{\argmax}{\mathop{\rm argmax}}




%% file: arxiv.bbl
\begin{thebibliography}{81}
\providecommand{\natexlab}[1]{#1}
\providecommand{\url}[1]{\texttt{#1}}
\expandafter\ifx\csname urlstyle\endcsname\relax
  \providecommand{\doi}[1]{doi: #1}\else
  \providecommand{\doi}{doi: \begingroup \urlstyle{rm}\Url}\fi

\bibitem[Abernethy et~al.(2018)Abernethy, Lai, Levy, and
  Wang]{abernethy2018faster}
J.~Abernethy, K.~A. Lai, K.~Y. Levy, and J.-K. Wang.
\newblock Faster rates for convex-concave games.
\newblock \emph{arXiv preprint arXiv:1805.06792}, 2018.

\bibitem[Abernethy and Wang(2017)]{abernethy2017frank}
J.~D. Abernethy and J.-K. Wang.
\newblock On frank-wolfe and equilibrium computation.
\newblock In \emph{Advances in Neural Information Processing Systems}, pages
  6584--6593, 2017.

\bibitem[Alayrac et~al.(2016)Alayrac, Bojanowski, Agrawal, Sivic, Laptev, and
  Lacoste-Julien]{alayrac2016unsupervised}
J.-B. Alayrac, P.~Bojanowski, N.~Agrawal, J.~Sivic, I.~Laptev, and
  S.~Lacoste-Julien.
\newblock Unsupervised learning from narrated instruction videos.
\newblock In \emph{Proceedings of the IEEE Conference on Computer Vision and
  Pattern Recognition}, pages 4575--4583, 2016.

\bibitem[Allen-Zhu et~al.(2017)Allen-Zhu, Hazan, Hu, and Li]{allen2017linear}
Z.~Allen-Zhu, E.~Hazan, W.~Hu, and Y.~Li.
\newblock Linear convergence of a frank-wolfe type algorithm over trace-norm
  balls.
\newblock In \emph{Advances in Neural Information Processing Systems}, pages
  6191--6200, 2017.

\bibitem[Bach et~al.(2012)Bach, Lacoste-Julien, and
  Obozinski]{bach2012equivalence}
F.~Bach, S.~Lacoste-Julien, and G.~Obozinski.
\newblock On the equivalence between herding and conditional gradient
  algorithms.
\newblock \emph{arXiv preprint arXiv:1203.4523}, 2012.

\bibitem[Balashov and Repovs(2011)]{balashov2011uniformly}
M.~V. Balashov and D.~Repovs.
\newblock Uniformly convex subsets of the hilbert space with modulus of
  convexity of the second order.
\newblock \emph{arXiv preprint arXiv:1101.5685}, 2011.

\bibitem[Ball et~al.(1994)Ball, Carlen, and Lieb]{ball1994sharp}
K.~Ball, E.~A. Carlen, and E.~H. Lieb.
\newblock Sharp uniform convexity and smoothness inequalities for trace norms.
\newblock \emph{Inventiones mathematicae}, 115\penalty0 (1):\penalty0 463--482,
  1994.

\bibitem[Beck and Shtern(2017)]{beck2017linearly}
A.~Beck and S.~Shtern.
\newblock Linearly convergent away-step conditional gradient for non-strongly
  convex functions.
\newblock \emph{Mathematical Programming}, 164\penalty0 (1-2):\penalty0 1--27,
  2017.

\bibitem[Boas~Jr(1940)]{boas1940some}
R.~P. Boas~Jr.
\newblock Some uniformly convex spaces.
\newblock \emph{Bulletin of the American Mathematical Society}, 46\penalty0
  (4):\penalty0 304--311, 1940.

\bibitem[Bojanowski et~al.(2014)Bojanowski, Lajugie, Bach, Laptev, Ponce,
  Schmid, and Sivic]{bojanowski2014weakly}
P.~Bojanowski, R.~Lajugie, F.~Bach, I.~Laptev, J.~Ponce, C.~Schmid, and
  J.~Sivic.
\newblock Weakly supervised action labeling in videos under ordering
  constraints.
\newblock In \emph{European Conference on Computer Vision}, pages 628--643.
  Springer, 2014.

\bibitem[Bojanowski et~al.(2015)Bojanowski, Lajugie, Grave, Bach, Laptev,
  Ponce, and Schmid]{bojanowski2015weakly}
P.~Bojanowski, R.~Lajugie, E.~Grave, F.~Bach, I.~Laptev, J.~Ponce, and
  C.~Schmid.
\newblock Weakly-supervised alignment of video with text.
\newblock In \emph{Proceedings of the IEEE international conference on computer
  vision}, pages 4462--4470, 2015.

\bibitem[Bolte et~al.(2007)Bolte, Daniilidis, and Lewis]{bolte2007lojasiewicz}
J.~Bolte, A.~Daniilidis, and A.~Lewis.
\newblock The {\l}ojasiewicz inequality for nonsmooth subanalytic functions
  with applications to subgradient dynamical systems.
\newblock \emph{SIAM Journal on Optimization}, 17\penalty0 (4):\penalty0
  1205--1223, 2007.

\bibitem[Canon and Cullum(1968)]{canon1968tight}
M.~D. Canon and C.~D. Cullum.
\newblock A tight upper bound on the rate of convergence of frank-wolfe
  algorithm.
\newblock \emph{SIAM Journal on Control}, 6\penalty0 (4):\penalty0 509--516,
  1968.

\bibitem[Clarkson(1936)]{clarkson1936uniformly}
J.~A. Clarkson.
\newblock Uniformly convex spaces.
\newblock \emph{Transactions of the American Mathematical Society}, 40\penalty0
  (3):\penalty0 396--414, 1936.

\bibitem[Clarkson(2010)]{Clar10}
K.~Clarkson.
\newblock Coresets, sparse greedy approximation, and the {F}rank-{W}olfe
  algorithm.
\newblock \emph{ACM Transactions on Algorithms (TALG)}, 6\penalty0
  (4):\penalty0 63, 2010.

\bibitem[Combettes and Pokutta(2019)]{combettes2019revisiting}
C.~W. Combettes and S.~Pokutta.
\newblock Revisiting the approximate carath\'eodory problem via the frank-wolfe
  algorithm.
\newblock \emph{arXiv preprint arXiv:1911.04415}, 2019.

\bibitem[Courty et~al.(2016)Courty, Flamary, Tuia, and
  Rakotomamonjy]{courty2016optimal}
N.~Courty, R.~Flamary, D.~Tuia, and A.~Rakotomamonjy.
\newblock Optimal transport for domain adaptation.
\newblock \emph{IEEE transactions on pattern analysis and machine
  intelligence}, 39\penalty0 (9):\penalty0 1853--1865, 2016.

\bibitem[Dekel et~al.(2017)Dekel, Haghtalab, Jaillet, et~al.]{dekel2017online}
O.~Dekel, N.~Haghtalab, P.~Jaillet, et~al.
\newblock Online learning with a hint.
\newblock In \emph{Advances in Neural Information Processing Systems}, pages
  5299--5308, 2017.

\bibitem[Demyanov and Rubinov(1970)]{demyanov1970}
V.~F. Demyanov and A.~M. Rubinov.
\newblock Approximate methods in optimization problems.
\newblock \emph{Modern Analytic and Computational Methods in Science and
  Mathematics}, 1970.

\bibitem[Dixmier(1953)]{dixmier1953formes}
J.~Dixmier.
\newblock Formes lin{\'e}aires sur un anneau d'op{\'e}rateurs.
\newblock \emph{Bulletin de la Soci{\'e}t{\'e} Math{\'e}matique de France},
  81:\penalty0 9--39, 1953.

\bibitem[Donahue et~al.(1997)Donahue, Darken, Gurvits, and
  Sontag]{donahue1997rates}
M.~J. Donahue, C.~Darken, L.~Gurvits, and E.~Sontag.
\newblock Rates of convex approximation in non-hilbert spaces.
\newblock \emph{Constructive Approximation}, 13\penalty0 (2):\penalty0
  187--220, 1997.

\bibitem[Dudik et~al.(2012)Dudik, Harchaoui, and Malick]{dudik2012lifted}
M.~Dudik, Z.~Harchaoui, and J.~Malick.
\newblock Lifted coordinate descent for learning with trace-norm
  regularization.
\newblock In \emph{Artificial Intelligence and Statistics}, pages 327--336,
  2012.

\bibitem[Dunn(1979)]{dunn1979rates}
J.~C. Dunn.
\newblock Rates of convergence for conditional gradient algorithms near
  singular and nonsingular extremals.
\newblock \emph{SIAM Journal on Control and Optimization}, 17\penalty0
  (2):\penalty0 187--211, 1979.

\bibitem[Frank and Wolfe(1956)]{frank1956algorithm}
M.~Frank and P.~Wolfe.
\newblock An algorithm for quadratic programming.
\newblock \emph{Naval research logistics quarterly}, 3\penalty0 (1-2):\penalty0
  95--110, 1956.

\bibitem[Freund et~al.(2017)Freund, Grigas, and Mazumder]{freund2017extended}
R.~M. Freund, P.~Grigas, and R.~Mazumder.
\newblock An extended frank--wolfe method with ``in-face'' directions, and its
  application to low-rank matrix completion.
\newblock \emph{SIAM Journal on Optimization}, 27\penalty0 (1):\penalty0
  319--346, 2017.

\bibitem[Futami et~al.(2019)Futami, Cui, Sato, and
  Sugiyama]{futami2019bayesian}
F.~Futami, Z.~Cui, I.~Sato, and M.~Sugiyama.
\newblock Bayesian posterior approximation via greedy particle optimization.
\newblock In \emph{Proceedings of the AAAI Conference on Artificial
  Intelligence}, volume~33, pages 3606--3613, 2019.

\bibitem[Garber and Hazan(2013{\natexlab{a}})]{garber2013linearly}
D.~Garber and E.~Hazan.
\newblock A linearly convergent conditional gradient algorithm with
  applications to online and stochastic optimization.
\newblock \emph{arXiv preprint arXiv:1301.4666}, 2013{\natexlab{a}}.

\bibitem[Garber and Hazan(2013{\natexlab{b}})]{garber2013playing}
D.~Garber and E.~Hazan.
\newblock Playing non-linear games with linear oracles.
\newblock In \emph{2013 IEEE 54th Annual Symposium on Foundations of Computer
  Science}, pages 420--428. IEEE, 2013{\natexlab{b}}.

\bibitem[Garber and Hazan(2015)]{garber2015faster}
D.~Garber and E.~Hazan.
\newblock Faster rates for the frank-wolfe method over strongly-convex sets.
\newblock In \emph{32nd International Conference on Machine Learning, ICML
  2015}, 2015.

\bibitem[Garber et~al.(2018)Garber, Sabach, and Kaplan]{garber2018fast}
D.~Garber, S.~Sabach, and A.~Kaplan.
\newblock Fast generalized conditional gradient method with applications to
  matrix recovery problems.
\newblock \emph{arXiv preprint arXiv:1802.05581}, 2018.

\bibitem[Goncharov and Ivanov(2017)]{goncharov2017strong}
V.~V. Goncharov and G.~E. Ivanov.
\newblock Strong and weak convexity of closed sets in a hilbert space.
\newblock In \emph{Operations research, engineering, and cyber security}, pages
  259--297. Springer, 2017.

\bibitem[Gu{\'e}lat and Marcotte(1986)]{guelat1986some}
J.~Gu{\'e}lat and P.~Marcotte.
\newblock Some comments on {Wolfe}'s ‘away step’.
\newblock \emph{Mathematical Programming}, 1986.

\bibitem[Gutman and Pena(2018)]{gutman2018condition}
D.~H. Gutman and J.~F. Pena.
\newblock The condition of a function relative to a polytope.
\newblock \emph{arXiv preprint arXiv:1802.00271}, 2018.

\bibitem[Hanner et~al.(1956)]{hanner1956}
O.~Hanner et~al.
\newblock On the uniform convexity of lp and lp.
\newblock \emph{Arkiv f{\"o}r Matematik}, 3\penalty0 (3):\penalty0 239--244,
  1956.

\bibitem[Harchaoui et~al.(2012)Harchaoui, Douze, Paulin, Dudik, and
  Malick]{harchaoui2012large}
Z.~Harchaoui, M.~Douze, M.~Paulin, M.~Dudik, and J.~Malick.
\newblock Large-scale image classification with trace-norm regularization.
\newblock In \emph{2012 IEEE Conference on Computer Vision and Pattern
  Recognition}, pages 3386--3393. IEEE, 2012.

\bibitem[Hazan and Kale(2012)]{hazan2012projection}
E.~Hazan and S.~Kale.
\newblock Projection-free online learning.
\newblock \emph{arXiv preprint arXiv:1206.4657}, 2012.

\bibitem[Hazan and Minasyan(2020)]{Hazan2020}
E.~Hazan and E.~Minasyan.
\newblock Faster projection-free online learning, 2020.

\bibitem[Hearn et~al.(1987)Hearn, Lawphongpanich, and
  Ventura]{hearn1987restricted}
D.~W. Hearn, S.~Lawphongpanich, and J.~A. Ventura.
\newblock Restricted simplicial decomposition: Computation and extensions.
\newblock In \emph{Computation Mathematical Programming}, pages 99--118.
  Springer, 1987.

\bibitem[Huang et~al.(2016)Huang, Lattimore, Gy{\"o}rgy, and
  Szepesv{\'a}ri]{huang2016following}
R.~Huang, T.~Lattimore, A.~Gy{\"o}rgy, and C.~Szepesv{\'a}ri.
\newblock Following the leader and fast rates in linear prediction: Curved
  constraint sets and other regularities.
\newblock In \emph{Advances in Neural Information Processing Systems}, pages
  4970--4978, 2016.

\bibitem[Huang et~al.(2017)Huang, Lattimore, Gy{\"o}rgy, and
  Szepesv{\'a}ri]{huang2017following}
R.~Huang, T.~Lattimore, A.~Gy{\"o}rgy, and C.~Szepesv{\'a}ri.
\newblock Following the leader and fast rates in online linear prediction:
  Curved constraint sets and other regularities.
\newblock \emph{The Journal of Machine Learning Research}, 18\penalty0
  (1):\penalty0 5325--5355, 2017.

\bibitem[Ivanov()]{ivanov2019approximate}
G.~Ivanov.
\newblock Approximate carath{\'e}odory’s theorem in uniformly smooth banach
  spaces.
\newblock \emph{Discrete \& Computational Geometry}, pages 1--8.

\bibitem[Jaggi(2011)]{Jagg11}
M.~Jaggi.
\newblock Convex optimization without projection steps.
\newblock \emph{Arxiv preprint arXiv:1108.1170}, 2011.

\bibitem[Jaggi(2013)]{jaggi2013revisiting}
M.~Jaggi.
\newblock Revisiting frank-wolfe: Projection-free sparse convex optimization.
\newblock In \emph{Proceedings of the 30th international conference on machine
  learning}, number CONF, pages 427--435, 2013.

\bibitem[Jaggi and Sulovsk{\`y}(2010)]{jaggi2010simple}
M.~Jaggi and M.~Sulovsk{\`y}.
\newblock A simple algorithm for nuclear norm regularized problems.
\newblock 2010.

\bibitem[Journ{\'e}e et~al.(2010)Journ{\'e}e, Nesterov, Richt{\'a}rik, and
  Sepulchre]{journee2010generalized}
M.~Journ{\'e}e, Y.~Nesterov, P.~Richt{\'a}rik, and R.~Sepulchre.
\newblock Generalized power method for sparse principal component analysis.
\newblock \emph{Journal of Machine Learning Research}, 11\penalty0
  (Feb):\penalty0 517--553, 2010.

\bibitem[Juditsky and Nemirovski(2008)]{juditsky2008large}
A.~Juditsky and A.~S. Nemirovski.
\newblock Large deviations of vector-valued martingales in 2-smooth normed
  spaces.
\newblock \emph{arXiv preprint arXiv:0809.0813}, 2008.

\bibitem[Kerdreux et~al.(2019)Kerdreux, d’Aspremont, and
  Pokutta]{kerdreux2019restarting}
T.~Kerdreux, A.~d’Aspremont, and S.~Pokutta.
\newblock Restarting frank-wolfe.
\newblock In \emph{The 22nd International Conference on Artificial Intelligence
  and Statistics}, pages 1275--1283, 2019.

\bibitem[Kurdyka(1998)]{kurdyka1998gradients}
K.~Kurdyka.
\newblock On gradients of functions definable in o-minimal structures.
\newblock In \emph{Annales de l'institut Fourier}, volume~48, pages 769--783,
  1998.

\bibitem[Lacoste-Julien and Jaggi(2013)]{lacoste2013affine}
S.~Lacoste-Julien and M.~Jaggi.
\newblock An affine invariant linear convergence analysis for frank-wolfe
  algorithms.
\newblock \emph{arXiv preprint arXiv:1312.7864}, 2013.

\bibitem[Lacoste-Julien and Jaggi(2015)]{lacoste2015global}
S.~Lacoste-Julien and M.~Jaggi.
\newblock On the global linear convergence of {F}rank--{W}olfe optimization
  variants.
\newblock In C.~Cortes, N.~D. Lawrence, D.~D. Lee, M.~Sugiyama, and R.~Garnett,
  editors, \emph{Advances in Neural Information Processing Systems}, volume~28,
  pages 496--504. Curran Associates, Inc., 2015.

\bibitem[Lacoste-Julien et~al.(2015)Lacoste-Julien, Lindsten, and
  Bach]{lacoste2015sequential}
S.~Lacoste-Julien, F.~Lindsten, and F.~Bach.
\newblock Sequential kernel herding: Frank-wolfe optimization for particle
  filtering.
\newblock \emph{arXiv preprint arXiv:1501.02056}, 2015.

\bibitem[Lafond et~al.(2015)Lafond, Wai, and Moulines]{lafond2015online}
J.~Lafond, H.-T. Wai, and E.~Moulines.
\newblock On the online frank-wolfe algorithms for convex and non-convex
  optimizations.
\newblock \emph{arXiv preprint arXiv:1510.01171}, 2015.

\bibitem[Lan(2013)]{lan2013complexity}
G.~Lan.
\newblock The complexity of large-scale convex programming under a linear
  optimization oracle.
\newblock \emph{arXiv preprint arXiv:1309.5550}, 2013.

\bibitem[Levitin and Polyak(1966)]{levitin1966constrained}
E.~S. Levitin and B.~T. Polyak.
\newblock Constrained minimization methods.
\newblock \emph{USSR Computational mathematics and mathematical physics},
  6\penalty0 (5):\penalty0 1--50, 1966.

\bibitem[Levy and Krause(2019)]{levy2019projection}
K.~Levy and A.~Krause.
\newblock Projection free online learning over smooth sets.
\newblock In \emph{The 22nd International Conference on Artificial Intelligence
  and Statistics}, pages 1458--1466, 2019.

\bibitem[Lindenstrauss and Tzafriri(2013)]{lindenstrauss2013classical}
J.~Lindenstrauss and L.~Tzafriri.
\newblock \emph{Classical Banach spaces II: function spaces}, volume~97.
\newblock Springer Science \& Business Media, 2013.

\bibitem[Lojasiewicz(1965)]{law1965ensembles}
S.~Lojasiewicz.
\newblock Ensembles semi-analytiques.
\newblock \emph{Institut des Hautes \'Etudes Scientifiques}, 1965.

\bibitem[Luise et~al.(2019)Luise, Salzo, Pontil, and
  Ciliberto]{luise2019sinkhorn}
G.~Luise, S.~Salzo, M.~Pontil, and C.~Ciliberto.
\newblock Sinkhorn barycenters with free support via frank-wolfe algorithm.
\newblock In \emph{Advances in Neural Information Processing Systems}, pages
  9318--9329, 2019.

\bibitem[Miech et~al.(2017)Miech, Alayrac, Bojanowski, Laptev, and
  Sivic]{miech2017learning}
A.~Miech, J.-B. Alayrac, P.~Bojanowski, I.~Laptev, and J.~Sivic.
\newblock Learning from video and text via large-scale discriminative
  clustering.
\newblock In \emph{Proceedings of the IEEE International Conference on Computer
  Vision}, pages 5257--5266, 2017.

\bibitem[Miech et~al.(2018)Miech, Laptev, and Sivic]{miech2018learning}
A.~Miech, I.~Laptev, and J.~Sivic.
\newblock Learning a text-video embedding from incomplete and heterogeneous
  data.
\newblock \emph{arXiv preprint arXiv:1804.02516}, 2018.

\bibitem[Molinaro(2020)]{Molinaro2020}
M.~Molinaro.
\newblock Curvature of feasible sets in offline and online optimization.
\newblock 2020.

\bibitem[Nguyen and Petrova(2017)]{nguyen2017greedy}
H.~Nguyen and G.~Petrova.
\newblock Greedy strategies for convex optimization.
\newblock \emph{Calcolo}, 54\penalty0 (1):\penalty0 207--224, 2017.

\bibitem[Paty and Cuturi(2019)]{paty2019subspace}
F.-P. Paty and M.~Cuturi.
\newblock Subspace robust wasserstein distances.
\newblock \emph{arXiv preprint arXiv:1901.08949}, 2019.

\bibitem[Pena and Rodriguez(2018)]{pena2018polytope}
J.~Pena and D.~Rodriguez.
\newblock Polytope conditioning and linear convergence of the frank--wolfe
  algorithm.
\newblock \emph{Mathematics of Operations Research}, 2018.

\bibitem[Peyre et~al.(2017)Peyre, Sivic, Laptev, and Schmid]{peyre2017weakly}
J.~Peyre, J.~Sivic, I.~Laptev, and C.~Schmid.
\newblock Weakly-supervised learning of visual relations.
\newblock In \emph{Proceedings of the IEEE International Conference on Computer
  Vision}, pages 5179--5188, 2017.

\bibitem[Polyak(1966)]{polyak1966existence}
B.~T. Polyak.
\newblock Existence theorems and convergence of minimizing sequences for
  extremal problems with constraints.
\newblock In \emph{Doklady Akademii Nauk}, volume 166, pages 287--290. Russian
  Academy of Sciences, 1966.

\bibitem[Rockafellar(1970)]{Rock70}
R.~T. Rockafellar.
\newblock \emph{Convex Analysis}.
\newblock Princeton University Press., Princeton., 1970.

\bibitem[Schneider(2014)]{schneider2014convex}
R.~Schneider.
\newblock \emph{Convex bodies: the Brunn--Minkowski theory}.
\newblock Number 151. Cambridge university press, 2014.

\bibitem[Schneider(2015)]{schneider2015curvatures}
R.~Schneider.
\newblock Curvatures of typical convex bodies—the complete picture.
\newblock \emph{Proceedings of the American Mathematical Society}, 143\penalty0
  (1):\penalty0 387--393, 2015.

\bibitem[Seguin et~al.(2016)Seguin, Bojanowski, Lajugie, and
  Laptev]{seguin2016instance}
G.~Seguin, P.~Bojanowski, R.~Lajugie, and I.~Laptev.
\newblock Instance-level video segmentation from object tracks.
\newblock In \emph{Proceedings of the IEEE Conference on Computer Vision and
  Pattern Recognition}, pages 3678--3687, 2016.

\bibitem[Shalev-Shwartz(2007)]{shai2007phd}
S.~Shalev-Shwartz.
\newblock \emph{Online Learning: Theory, Algorithms, and Applications}.
\newblock PhD thesis, 2007.

\bibitem[Shalev-Shwartz et~al.(2011)Shalev-Shwartz, Gonen, and
  Shamir]{shalev2011large}
S.~Shalev-Shwartz, A.~Gonen, and O.~Shamir.
\newblock Large-scale convex minimization with a low-rank constraint.
\newblock \emph{arXiv preprint arXiv:1106.1622}, 2011.

\bibitem[Shalev-Shwartz et~al.(2012)]{shalev2012online}
S.~Shalev-Shwartz et~al.
\newblock Online learning and online convex optimization.
\newblock \emph{Foundations and Trends{\textregistered} in Machine Learning},
  4\penalty0 (2):\penalty0 107--194, 2012.

\bibitem[Simon et~al.(1979)]{simon1979trace}
B.~Simon et~al.
\newblock Trace ideals and their applications.
\newblock 1979.

\bibitem[So(1990)]{so1990facial}
W.~So.
\newblock Facial structures of schatten p-norms.
\newblock \emph{Linear and Multilinear Algebra}, 27\penalty0 (3):\penalty0
  207--212, 1990.

\bibitem[Temlyakov(2011)]{temlyakov2011greedy}
V.~Temlyakov.
\newblock \emph{Greedy approximation}, volume~20.
\newblock Cambridge University Press, 2011.

\bibitem[Tomczak-Jaegermann(1974)]{tomczak1974moduli}
N.~Tomczak-Jaegermann.
\newblock The moduli of smoothness and convexity and the rademacher averages of
  the trace classes.
\newblock \emph{Studia Mathematica}, 50\penalty0 (2):\penalty0 163--182, 1974.

\bibitem[Vayer et~al.(2018)Vayer, Chapel, Flamary, Tavenard, and
  Courty]{vayer2018optimal}
T.~Vayer, L.~Chapel, R.~Flamary, R.~Tavenard, and N.~Courty.
\newblock Optimal transport for structured data with application on graphs.
\newblock \emph{arXiv preprint arXiv:1805.09114}, 2018.

\bibitem[Vial(1983)]{vial1983strong}
J.-P. Vial.
\newblock Strong and weak convexity of sets and functions.
\newblock \emph{Mathematics of Operations Research}, 8\penalty0 (2):\penalty0
  231--259, 1983.

\bibitem[Weber and Reisig(2013)]{weber2013local}
A.~Weber and G.~Reisig.
\newblock Local characterization of strongly convex sets.
\newblock \emph{Journal of Mathematical Analysis and Applications},
  400\penalty0 (2):\penalty0 743--750, 2013.

\bibitem[Xu and Yang(2018)]{YiAdapativeFW}
Y.~Xu and T.~Yang.
\newblock Frank-wolfe method is automatically adaptive to error bound
  condition, 2018.

\end{thebibliography}
